\documentclass[12pt,reqno]{amsart}
\usepackage{amsmath}
\usepackage{amssymb}
\usepackage{amstext}
\usepackage{mathrsfs}
\usepackage{a4wide}
\usepackage{graphicx}
\usepackage{bbm}

\allowdisplaybreaks \numberwithin{equation}{section}
\usepackage{color}
\usepackage{cases}

\numberwithin{equation}{section}

\newtheorem{theorem}{Theorem}[section]

\newtheorem{corollary}[theorem]{Corollary}
\newtheorem{lemma}[theorem]{Lemma}

\theoremstyle{definition}

\theoremstyle{remark}
\newtheorem{remark}[theorem]{Remark}

\begin{document}

\title
{Existence and regularity of co-rotating and travelling global solutions for the generalized SQG equation}

\author{Daomin Cao, Guolin Qin,  Weicheng Zhan, Changjun Zou}

\address{Institute of Applied Mathematics, Chinese Academy of Sciences, Beijing 100190, and University of Chinese Academy of Sciences, Beijing 100049,  P.R. China}
\email{dmcao@amt.ac.cn}
\address{Institute of Applied Mathematics, Chinese Academy of Sciences, Beijing 100190, and University of Chinese Academy of Sciences, Beijing 100049,  P.R. China}
\email{qinguolin18@mails.ucas.edu.cn}
\address{Institute of Applied Mathematics, Chinese Academy of Sciences, Beijing 100190, and University of Chinese Academy of Sciences, Beijing 100049,  P.R. China}
\email{zhanweicheng16@mails.ucas.ac.cn}

\address{Institute of Applied Mathematics, Chinese Academy of Sciences, Beijing 100190, and University of Chinese Academy of Sciences, Beijing 100049,  P.R. China}
\email{zouchangjun17@mails.ucas.ac.cn}

%\thanks{This work is partially supported by ARC}

\begin{abstract}
	By studying the linearization of contour dynamics equation and using implicit function theorem, we prove the existence of co-rotating and travelling global solutions for the gSQG equation, which extends the result of Hmidi and Mateu \cite{HM} to $\alpha\in[1,2)$. Moreover, we prove the $C^\infty$ regularity of vortices boundary, and show the convexity of each vortices component.
\end{abstract}

\maketitle

\section{Introduction and main results}
We will consider the generalized surface quasi-geostrophic (gSQG) equation
\begin{align}\label{1-1}
	\begin{cases}
		\partial_t\vartheta+\mathbf{v}\cdot \nabla \vartheta =0&\text{in}\ \mathbb{R}^2\times (0,T)\\
		\ \mathbf{v}=\nabla^\perp(-\Delta)^{-1+\frac{\alpha}{2}}\vartheta     &\text{in}\ \mathbb{R}^2\times (0,T),\\
		\vartheta\big|_{t=0}=\vartheta_0 &\text{in}\ \mathbb{R}^2\\
	\end{cases}
\end{align}
where $ 0\le\alpha<2$, $\vartheta(\boldsymbol x,t):\mathbb{R}^2\times (0,T)\to \mathbb{R}$ is the active scaler being transported by the velocity field $\mathbf{v}(\boldsymbol x,t):\mathbb{R}^2\times (0,T)\to \mathbb{R}^2$ generated by $\vartheta$, and $(x_1,x_2)^\perp=(x_2,-x_1)$. The operator $(-\Delta)^{-1+\frac{\alpha}{2}}$ is given by the expression
\begin{equation*}
	(-\Delta)^{-1+\frac{\alpha}{2}}\vartheta(\boldsymbol x)=\int_{\mathbb{R}^2}K_\alpha(\boldsymbol x-\boldsymbol y)\vartheta(\boldsymbol y)d\boldsymbol y,
\end{equation*}
where $K_\alpha$ is the fundamental solution of $(-\Delta)^{-1+\frac{\alpha}{2}}$ in $\mathbb{R}^2$ given by
\begin{equation*}
	K_\alpha( \ \cdot \ )=\left\{
	\begin{array}{lll}
		\frac{1}{2\pi}\ln \frac{1}{|\ \cdot \ |}, \ \ \ \ \ \ \ \ \ \ \ \ \ \ \ \ \ \ \ \ \ & \text{if} \ \ \alpha=0;\\
		\frac{C_\alpha}{2\pi}\frac{1}{|\ \cdot \ |^\alpha}, \ \ \ C_\alpha=\frac{\Gamma(\alpha/2)}{2^{1-\alpha}\Gamma(\frac{2-\alpha}{2})}, & \text{if} \ \ \ 0<\alpha<1.
	\end{array}
	\right.
\end{equation*}
with $\Gamma( \ \cdot \ )$ the Euler gamma function.

When $\alpha=0$, \eqref{1-1} is the vorticity formulation of 2D incompressible Euler equation. When $\alpha=1$, \eqref{1-1} is the surface quasi-geostrophic (SQG) equation, which is applied to study the atmosphere circulation and ocean dynamics. The gSQG model \eqref{1-1} $0<\alpha<2$ was proposed by C\'{o}rdoba et al. in \cite{Cor}, and was intensively investigated In the past decade as a generalization of the Euler equation and the SQG equation.

For the case $\alpha=0$, Yudovich \cite{Yud} proved the global well-posedness of \eqref{1-1} with the intial data in $L^1\cap L^\infty$ in 1960s. However, it is delicate to extend this theory to the general case $0<\alpha<2$, since the velocity is singular and scales below the Lipschitz class. Constantin et al. \cite{Con} established local well-posedness of the gSQG equation for classical solutions, and this property is known for sufficiently regular intial data as in \cite{Chae,Gan,Kis2}. local existence of solutions was also studied in different function spaces, see \cite{Chae0,Li0,Wu1,Wu2}.  An interesting issue for the gSQG model is the finte time below up versus global existence of solutions. In \cite{Kis}, Kiselev an Nazarov constructed solutions of the gSQG equations with arbitrary Sobolev growth. For the later situation, Resnick \cite{Res} proved global existence of weak solutions to the SQG equations with any initial data in $L^2$, which was improved by Marchand \cite{Mar} to any initial data belonging to $L^p$ with $p>4/3$. On the other hand, various of global solutions to \eqref{1-1} were constructed, which is the topic we are mainly concerned with here.

Due to the structure of the nonlinear term, all radially symmetric functions $\vartheta$ are stationary solutions to the gSQG equation. This fact inspires mathematicians to construct other global solutions which do not change form as time evolves. There are two kinds of solutions of this type: rotating solutions with uniform angular velocity, and travelling solution pairs with uniform speed. 

To talk about rotating solutions, we recall that a domain is said to be $m$-fold symmetric if it is invariant under the rotation around its center with an angular $2\pi/m$. The first explicit non-trival rotating solution of $2$-fold symmetry for $\alpha=1$ is constructed by Kirchhoff \cite{Kir}, which is an elliptic patch of semi-axes $a$ and $b$ subjected to a perpetual rotation with uniform angular velocity $ab/(a+b)^2$. Deem and Zabusky \cite{Deem} then carried out a series of numerical simulations, which provided evidence for existence of various $m$-fold symmetric solutions with $m\ge2$. In 1980s, Burbea \cite{Burb} put forward a new way to construct $m$-fold symmetric patch solutions by bifurcation theory. This approach is based on the linearization of the contour dynamic equation at trival solutions, and was greatly developed nowadays. We will discuss this approach in detail later. There are also other ways for the same purpose. Turkington \cite{T2} proved the existence of co-rotating $m$-fold symmetric vortex paches for the case $\alpha=0$ by a dual variational principle, which was extended by Godard-Cadillac et al. \cite{Go} to the gSQG model with $\alpha\in(0,1]$. Except for the foresaid results, Ao et al. \cite{Ao} gave the construction of smooth solutions for all $\alpha\in(0,2)$ by the Lyapunov-Schmidt reduction.

The study of travelling solution pairs can be traced back to Pocklington \cite{Poc}. In \cite{Lamb}, Lamb gave an explicit example of travelling vortex pairs in $\alpha=0$ which is now generally referred to as the Lamb dipole or Chaplygin-Lamb dipole. In the 3D axisymmetric Euler flow, the corresponding phenomenon of Lamb dipole is known as the Hill vortex. For vortex pairs whose two components are close, which are near the Lamb dipole, the bifurcation method is the mainstream to derive existence, see \cite{Amc}. While if the vorticity is restricted in two far separated and axisymmetric small domains, vortex pairs share a same construction with the co-rotating solutions, and we refer to \cite{Ao,Cao4,Go0,Gra} for more discussion. 

The main tool in this paper is the contour dynamics equation for $\alpha$-patch, for which we will give a brief introduction. An $\alpha$-patch is a solution of \eqref{1-1} whose initial data is given by $\vartheta_0=\boldsymbol \chi_{D}$, with $D\subset \mathbb{R}^2$ a bounded domain and $\boldsymbol \chi_{D}$ its characteristic function. Uniformly rotating rotating $\alpha$-patches are also known as $V$-states. Due to the transport formula $\partial_t\vartheta+\mathbf{v}\cdot \nabla \vartheta =0$, the solution will preserve its patch structure and can be written as $\vartheta_t=\boldsymbol \chi_{D_t}$. When $0<\alpha<1$, using Biot-Savart law and Green-Stokes formula, the velocity can be recovered by
\begin{equation*}
	\boldsymbol v(\boldsymbol x,t)=\frac{C_\alpha}{2\pi}\int_{\partial D_t}\frac{1}{|x-\xi|^\alpha}d\xi.
\end{equation*} 
Thus if the patch boundary $\partial D_t$ is parameterized as $\boldsymbol z(t,\sigma)$ with $\sigma\in [0,2\pi)$, then $\boldsymbol z(t,\sigma)$ will satisfy
\begin{equation*}
	\partial_t \boldsymbol z(t,\sigma)=\frac{C_\alpha}{2\pi}\int_0^{2\pi}\frac{\partial_\tau\boldsymbol z(t,\tau)}{|\boldsymbol z(t,\sigma)-\boldsymbol z(t,\tau)|^\alpha}d\tau,
\end{equation*}
which is known as the contour dynamics equation. When $1\le\alpha<2$, the above integral is divergent. To eliminate the singularity, we can substract a tangential vector in above integral and define
\begin{equation*}
	\partial_t \boldsymbol z(t,\sigma)=\frac{C_\alpha}{2\pi}\int_0^{2\pi}\frac{\partial_\tau\boldsymbol z(t,\tau)-\partial_\sigma\boldsymbol z(t,\sigma)}{|\boldsymbol z(t,\sigma)-\boldsymbol z(t,\tau)|^\alpha}d\tau.
\end{equation*}
The contour dynamics equation is locally well-posed if the boundary of the initial $\alpha$-patch is composed with finite number of disjoint smooth Jordan curves.

Following the spirit of Burbea in \cite{Burb}, Hassainia and Hmidi \cite{Has} used the contour dynamics equation to construct simply-connected $m$-fold symmetric V-states for $\alpha\in[0,1)$, which are bifurcated from the unit disk. Their construction relys on the structure of linearization of contour dynamics equation, and the uniform angular velocity is selected appropriately such that the transversality assumption of Crandall-Rabinowitz's theorem is satisfied. By changing the function spaces from H\"older space to $H^k$ space on torus, Castro et al. \cite{Cas1} proved the existence of solutions of same type for the remaining open cases $\alpha\in[1,2)$.  In \cite{de1,de2,de3}, the existence of doublely-connected V-states are established by studying the bifurcation from anulus with special angular velocity. However, there are two coupled nolinear equations in this situation, and the estimate for the spectral of linearized operator is much more difficult. In \cite{Cas4,Hmi2}, V-states bifurcated from Kirchhoff elliptic vortices were considered for $\alpha=0$. Recently, Castro et al. \cite{Cas3,Cas2} also successfully applied this method to construct uniformly rotating smooth solutions. 

Except for Crandall-Rabinowitz's theorem, another method to construct global solutions by contour dynamics equation is using implicit function theorem. Different from solutions obtained by bifurcation which are somehow scattered, the vorticity constructed by implicit function theorem is located in several small domains, and can be regarded as an approximation of point vortices. By this method, Hmidi and Mateu \cite{HM} gave a direct proof of the existence of co-rotating and travelling concentrating patch pairs for $\alpha\in [0,1)$. Carc\'ia \cite{Gar2} considered the $m$-fold symmetric case with the same approach. In this paper, we will close the question of the existence of co-rotating and travelling global patch solutions for $\alpha\in[1,2)$ by introducing function spaces as in \cite{Cas1}.

To state our main results, we first fix some notations which will be frequently used in this paper. For $\varepsilon\in [0,1/2)$, we denote $D^\varepsilon_0$ as a simply connected domain containing the origin, which is close to the ball with radius $\varepsilon$ centered at the origin. $\partial D^\varepsilon_0$ can be parameterized as
$$\boldsymbol{z}(x)=\left(\varepsilon R(x)\cos(x), \varepsilon R(x)\sin(x)\right),$$
where $$R(x)=1+\varepsilon^{1+\alpha}f(x)$$
with $x\in [0,2\pi)$, and $f$ some $C^1$ function. We will use $\boldsymbol\chi_D$ to denote the characteristic function of domain $D\in \mathbb{R}^2$, and $Q_\theta$ to denote the counterclockwise rotation operator of angle $\theta$ with respect to the origin. For simplicity reason, we let
$$\int\!\!\!\!\!\!\!\!\!\; {}-{} g(\tau)d\tau:=\frac{1}{2\pi}\int_0^{2\pi}g(\tau)d\tau$$ 
be the mean value of integral on the unit circle. The function spaces which we will use in this paper are 
\begin{equation*}
	X^k=\left\{ g\in H^k, \ g(x)= \sum\limits_{j=2}^{\infty}a_j\cos(jx)\right\},
\end{equation*}
\begin{equation*}	
		X^{k}_{\log}=\left\{ g\in H^k, \ g(x)= \sum\limits_{j=2}^{\infty}a_j\cos(jx), \ \left\|\int_0^{2\pi}\frac{\partial^kg(x-y)-\partial^kg(x)}{|\sin(\frac{y}{2})|}dy\right\|_{L^2}<\infty \right\},
\end{equation*}
\begin{equation*}
	Y^k=\left\{ g\in H^k, \ g(x)= \sum\limits_{j=1}^{\infty}a_j\sin(jx)\right\},
\end{equation*}
and
\begin{equation*}
	Y_0^k=Y^k/\text{span}\{\sin(x)\}=\left\{ g\in H^k, \ g(x)= \sum\limits_{j=2}^{\infty}a_j\sin(jx)\right\}.
\end{equation*}
It is easy to see that for every $\mu>0$, the embedding $X^{k+\mu}\subset X^{k}_{\log}\subset X^{k}$ holds.

The first kind of global solutions we will study are co-rotating solutions. Suppose the initial data $\vartheta_{0,\varepsilon}$ is an $m$-fold symmetric patch, that is 
\begin{equation}\label{1-2}
	\vartheta_{0,\varepsilon}(\boldsymbol x)=\frac{1}{\varepsilon^2}\sum\limits_{i=0}^{m-1}\boldsymbol\chi_{D^\varepsilon_i},
\end{equation}
where $D_i^\varepsilon\subset \mathbb{R}^2$ are domains satisfy
\begin{equation*}
	D_i^\varepsilon-d\boldsymbol e_1=Q_{\frac{2\pi i}{m}}\left(D_0^\varepsilon-d\boldsymbol e_1\right)
\end{equation*}
with some $d>1$ fixed, and $\boldsymbol e_1$ the unit vector in $x_1$ direction. We intend to prove the existence of a series of co-rotating $m$-fold symmetric solutions to \eqref{1-1} about $(d,0)$, which take the form 
\begin{equation*}
	\vartheta_\varepsilon(\boldsymbol x-d\boldsymbol e_1,t)=\vartheta_{0,\varepsilon}\left(Q_{\Omega t}(\boldsymbol x-d\boldsymbol e_1)\right)
\end{equation*}
with $\Omega$ some uniform angular velocity. If we combine this equality with \eqref{1-1}, we derive
\begin{equation*}
	\left(\mathbf{v}_0(\boldsymbol x)+\Omega(\boldsymbol x-d\boldsymbol e_1)^\perp\right)\cdot \nabla \vartheta_{0,\varepsilon}(\boldsymbol x)=0.
\end{equation*}
Then we can use the patch structure and obtain
\begin{equation*}
	\left(\mathbf{v}_0(\boldsymbol x)+\Omega(\boldsymbol x-d\boldsymbol e_1)^\perp\right)\cdot \mathbf n(\boldsymbol x)=0, \ \ \ \forall \, \boldsymbol x \in \cup_{i=0}^{m-1}\partial D^\varepsilon_i,
\end{equation*}
where $\mathbf n(\boldsymbol x)$ is the normal vector to the boundary. According to Biot-Savart law and Green-Stokes formula, it is sufficient to find $R(x)=1+\varepsilon^{1+\alpha}f(x)$, such that
\begin{equation}\label{1-3}
	\begin{split}
		& \Omega\left(\varepsilon R'(x)-\frac{dR'(x)\cos(x)}{R(x)}+d\sin(x)\right)\\
		&+\frac{C_\alpha}{\varepsilon^{1+\alpha}R(x)}\int\!\!\!\!\!\!\!\!\!\; {}-{} \frac{\left((R(x)R(y)+R'(x)R'(y))\sin(x-y)+(R(x)R'(y)-R'(x)R(y))\cos(x-y)\right)dy}{\left| \left(R(x)-R(y)\right)^2+4R(x)R(y)\sin^2\left(\frac{x-y}{2}\right)\right|^{\frac{\alpha}{2}}}+\\
		&\sum_{i=1}^{m-1} \frac{C_\alpha}{\varepsilon R(x)} \int\!\!\!\!\!\!\!\!\!\; {}-{} \frac{\left((R(x)R(y)+R'(x)R'(y))\sin(x-y-\frac{2\pi i}{m})+(R(x)R'(y)-R'(x)R(y))\cos(x-y-\frac{2\pi i}{m})\right)dy}{\left| \left(\boldsymbol{z}(x)-(d,0)\right)-Q_{\frac{2\pi i}{m}}\left(\boldsymbol{z}(y)-(d,0)\right)\right|^\alpha}\\\
		&=0.
	\end{split}
\end{equation}
In Section 2, we will show that the structure of \eqref{1-3} allows us to use implicit function theorem at $(\varepsilon,\Omega, f)=(0,\Omega_\alpha^*, 0)$ with
\begin{equation}\label{1-4}
	\Omega_\alpha^*:=\sum_{i=1}^{m-1} \frac{\alpha C_\alpha(-1+\cos(\frac{2\pi i}{m}))}{2\left((-1+\cos(\frac{2\pi i}{m}))^2 +\sin^2(\frac{2\pi i}{m})\right)^{1+\frac{\alpha}{2}}d^{2+\alpha}},
\end{equation}
and thus prove the existence of a family of co-rotating global solutions to \eqref{1-1} generated from $(0,\Omega_\alpha^*, 0)$. Actually, one can verify that $(0,\Omega_\alpha^*, 0)$ corresponds to co-rotating $m$-fold point vortices solutions to \eqref{1-1}, where the intensity of every single point vortex is $\pi$. Our first result can be stated as follow:
\begin{theorem}\label{thm1}
Suppose $\alpha\in [1,2)$, and $m\ge 2$. There exists $\varepsilon_0>0$ such that for any $\varepsilon\in [0,\varepsilon_0)$, \eqref{1-1} has a global co-rotating solution $\vartheta_\varepsilon(\boldsymbol x-d\boldsymbol e_1,t)=\vartheta_{0,\varepsilon}(Q_{\Omega_\alpha t}(\boldsymbol x-d\boldsymbol e_1))$ centered at $(d,0)$, where $\vartheta_{0,\varepsilon}$ is defined in \eqref{1-2}, and $\Omega_\alpha$ satisfies
$$\Omega_\alpha=\Omega_\alpha^*+O(\varepsilon^\alpha)$$
with $\Omega_\alpha^*$ given in \eqref{1-4}.
\end{theorem}

Then we shall move to consider the travelling patch pairs. In this situation, the initial data $\vartheta_{0,\varepsilon}$ is given by
\begin{equation}\label{1-5}
	\vartheta_{0,\varepsilon}(\boldsymbol x)=\frac{1}{\varepsilon^2}\boldsymbol\chi_{D^\varepsilon_0}-\frac{1}{\varepsilon^2}\boldsymbol\chi_{D^\varepsilon_T}.
\end{equation}
where $D_T^\varepsilon\subset \mathbb{R}^2$ satisfies
\begin{equation*}
	D_T^\varepsilon=-D_0^\varepsilon+2d\boldsymbol e_1
\end{equation*}
with some $d>1$ fixed.  A travelling patch pair centered at $(d,0)$ takes the form 
\begin{equation*}
	\vartheta_\varepsilon(\boldsymbol x,t)=\vartheta_{0,\varepsilon}(\boldsymbol x-tW\boldsymbol e_2)
\end{equation*}
with $W$ some fixed speed, and $\boldsymbol e_2$ the unit vector in $x_2$ direction. According to \eqref{1-1}, we derive
\begin{equation*}
	(\mathbf{v}_0(\boldsymbol x)-W\boldsymbol e_2)\cdot \nabla \vartheta_{0,\varepsilon}(\boldsymbol x)=0,
\end{equation*}
which yields
\begin{equation*}
	(\mathbf{v}_0(\boldsymbol x)-W\boldsymbol e_2)\cdot \mathbf n(\boldsymbol x)=0, \ \ \ \forall \, \boldsymbol x \in \partial D^\varepsilon_0\cup \partial D^\varepsilon_T.
\end{equation*}
Then the problem is reduced to finding $R(x)=1+\varepsilon^{1+\alpha}f(x)$, such that
\begin{equation}\label{1-6}
	\begin{split}
		&-W\left(\sin(x)-\frac{R'(x)}{R(x)}\cos(x)\right)\\
		&+\frac{C_\alpha}{\varepsilon^{1+\alpha}R(x)}\int\!\!\!\!\!\!\!\!\!\; {}-{} \frac{\left((R(x)R(y)+R'(x)R'(y))\sin(x-y)+(R(x)R'(y)-R'(x)R(y))\cos(x-y)\right)dy}{\left| \left(R(x)-R(y)\right)^2+4R(x)R(y)\sin^2\left(\frac{x-y}{2}\right)\right|^{\frac{\alpha}{2}}}\\
		&+\frac{C_\alpha}{\varepsilon R(x)} \int\!\!\!\!\!\!\!\!\!\; {}-{} \frac{\left((R(x)R(y)+R'(x)R'(y))\sin(x-y)+(R(x)R'(y)-R'(x)R(y))\cos(x-y)\right)dy}{\left|(\varepsilon R(x)\cos(x)+\varepsilon R(y)\cos(y)-2d)^2+(\varepsilon R(x)\sin(x)+\varepsilon R(y)\sin(y))^2\right|^{\frac{\alpha}{2}}}\\\
		&=0.
	\end{split}
\end{equation}
In Section 3, we will use implicit function theorem at $(\varepsilon, W, f)=(0,W_\alpha^*, 0)$ with
\begin{equation}\label{1-7}
	W_\alpha^*:= \frac{\alpha C_\alpha}{2(2d)^{1+\alpha}},
\end{equation}
and prove the existence of a family of travelling global solutions to \eqref{1-1} generated from $(0, W_\alpha^*, 0)$. One can also verify that $(0,W_\alpha^*, 0)$ corresponds to travelling point vortex pairs with the form $\vartheta_\varepsilon^*(\boldsymbol x,t)=\pi\boldsymbol\delta_{(0,tW_\alpha^*)}-\pi\boldsymbol\delta_{(2d,tW_\alpha^*)}$. By now, we can state our second result. 
\begin{theorem}\label{thm2}
	Suppose $\alpha\in [1,2)$. There exists $\varepsilon_0>0$ such that for any $\varepsilon\in [0,\varepsilon_0)$, \eqref{1-1} has a global travelling solution pair $\vartheta_\varepsilon(\boldsymbol x,t)=\vartheta_{0,\varepsilon}(\boldsymbol x-tW_\alpha\boldsymbol e_2)$ in $x_2$ direction centered at $(d,0)$, where $\vartheta_{0,\varepsilon}$ is defined in \eqref{1-5}, and $W_\alpha$ satisfies
	$$W_\alpha=W_\alpha^*+O(\varepsilon^\alpha)$$
	with $W_\alpha^*$ given in \eqref{1-7}.
\end{theorem}

In the study of global solutions to \eqref{1-1}, the regularity of vortices boundary is another challenging problem. In \cite{Che}, Chemin applied paradifferential calculus to show $C^{1,\gamma}$ regularity of patch boundary in the case $\alpha=0$. Then Bertozzi and Constantin \cite{Ber} obtained the same result using a different approach. In resent years, the contour dynamics equation was found to be an effective tool to investigate the boundary regularity.  Hmidi et al. \cite{Hmi} proved the $C^\infty$ regularity for simply-connected V-states in $\alpha=0$, and Castro et al. \cite{Cas1} generalized this result to $\alpha\in (0,2)$ by bootstrap. We will follow the strategy in \cite{Cas1} and prove following result for co-rotating and travelling global solutions.
\begin{theorem}\label{thm3}
	Suppose $\alpha\in [1,2)$ and $R(x)$ be a solution to \eqref{1-3} or \eqref{1-6}. Then $R(x)$ belongs to $C^\infty$ and parameterizes a convex patch. 
\end{theorem}
\begin{remark}
For the case $a\in [0,1)$, one can also use this method to prove the $C^\infty$ regularity of boundary.
\end{remark}

This paper is organized as follows. In section 2, we prove the existence of co-rotating patch solutions with $m$-fold symmetry. In Section 3, we investigate the existence of travelling patch pair solutions. In Section 4, we finish our paper by proving the $C^\infty$ regularity of the patch boundary, and the convexity of $D_0^\varepsilon$. Some auxiliary results will be given in the appendix.

\section{Existence of co-rotating global solutions for the gSQG equation }

Suppose $\varepsilon\in (-\frac{1}{2},\frac{1}{2})$, and let $R(x)=1+\varepsilon|\varepsilon|^\alpha f(x)$. According to \eqref{1-3}, we will consider $G^\alpha(\varepsilon, \Omega, f)$ with the form
\begin{equation*}
	G^\alpha(\varepsilon, \Omega, f)=G_1+G_2+G_3,
\end{equation*}
where
\begin{equation}\label{2-1}
	G_1=\Omega\left(|\varepsilon|^{2+\alpha} f'(x)-\frac{\varepsilon|\varepsilon|^\alpha f'(x)d\cos(x)}{1+\varepsilon|\varepsilon|^\alpha f(x)}+d\sin(x)\right),
\end{equation}

\begin{equation}\label{2-2}
	\begin{split}
		G_2&=\frac{C_\alpha}{\varepsilon|\varepsilon|^{\alpha}}\int\!\!\!\!\!\!\!\!\!\; {}-{} \frac{(1+\varepsilon|\varepsilon|^\alpha f(y))\sin(x-y)dy}{\left( |\varepsilon|^{2+2\alpha}\left(f(x)-f(y)\right)^2+4(1+\varepsilon|\varepsilon|^\alpha f(x))(1+\varepsilon|\varepsilon|^\alpha f(y))\sin^2\left(\frac{x-y}{2}\right)\right)^{\frac{\alpha}{2}}}\\
		&+C_\alpha\int\!\!\!\!\!\!\!\!\!\; {}-{} \frac{(f'(y)-f'(x))\cos(x-y)dy}{\left( |\varepsilon|^{2+2\alpha}\left(f(x)-f(y)\right)^2+4(1+\varepsilon|\varepsilon|^\alpha f(x))(1+\varepsilon|\varepsilon|^\alpha f(y))\sin^2\left(\frac{x-y}{2}\right)\right)^{\frac{\alpha}{2}}}\\		
		&+\frac{C_\alpha f'(x)}{1+\varepsilon|\varepsilon|^\alpha f(x)}\int\!\!\!\!\!\!\!\!\!\; {}-{} \frac{(f(x)-f'(y))\cos(x-y)dy}{\left( |\varepsilon|^{2+2\alpha}\left(f(x)-f(y)\right)^2+4(1+\varepsilon|\varepsilon|^\alpha f(x))(1+\varepsilon|\varepsilon|^\alpha f(y))\sin^2\left(\frac{x-y}{2}\right)\right)^{\frac{\alpha}{2}}}\\	
		&+\frac{C_\alpha\varepsilon|\varepsilon|^\alpha}{1+\varepsilon|\varepsilon|^\alpha f(x)}\int\!\!\!\!\!\!\!\!\!\; {}-{} \frac{f'(x)f'(y)\sin(x-y)dy}{\left( |\varepsilon|^{2+2\alpha}\left(f(x)-f(y)\right)^2+4(1+\varepsilon|\varepsilon|^\alpha f(x))(1+\varepsilon|\varepsilon|^\alpha f(y))\sin^2\left(\frac{x-y}{2}\right)\right)^{\frac{\alpha}{2}}}\\	
		&=G_{21}+G_{22}+G_{23}+G_{24},
	\end{split}
\end{equation}
and
\begin{equation}\label{2-3}
		\begin{split}
		G_3&=\sum\limits_{i=1}^{m-1}\frac{C_\alpha}{\varepsilon}\int\!\!\!\!\!\!\!\!\!\; {}-{} \frac{(1+\varepsilon|\varepsilon|^\alpha f(y))\sin(x-y-\frac{2\pi i}{m})dy}{\left| \left(\boldsymbol{z}(x)-(d,0)\right)-Q_{\frac{2\pi i}{m}}\left(\boldsymbol{z}(y)-(d,0)\right)\right|^\alpha}\\
		&+\sum\limits_{i=1}^{m-1}\frac{C_\alpha}{(1+\varepsilon|\varepsilon|^\alpha f(x))}\int\!\!\!\!\!\!\!\!\!\; {}-{} \frac{\varepsilon |\varepsilon|^{2\alpha}f'(x)f'(y)\sin(x-y-\frac{2\pi i}{m})dy}{\left| \left(\boldsymbol{z}(x)-(d,0)\right)-Q_{\frac{2\pi i}{m}}\left(\boldsymbol{z}(y)-(d,0)\right)\right|^\alpha}\\
		&+\sum\limits_{i=1}^{m-1}\frac{C_\alpha}{(1+\varepsilon|\varepsilon|^\alpha f(x))}\int\!\!\!\!\!\!\!\!\!\; {}-{} \frac{(|\varepsilon|^\alpha f(x)f'(y)-|\varepsilon|^\alpha f(x)f'(y))\cos(x-y-\frac{2\pi i}{m})dy}{\left| \left(\boldsymbol{z}(x)-(d,0)\right)-Q_{\frac{2\pi i}{m}}\left(\boldsymbol{z}(y)-(d,0)\right)\right|^\alpha}\\
		&=G_{31}+G_{32}+G_{33}
	\end{split}
\end{equation}

To show that $G^\alpha(\varepsilon, \Omega, f)$ satisfies the conditions for the implicit function theorem, we will verify that $G^\alpha(\varepsilon, \Omega, f)$ and $\partial_fG^\alpha(\varepsilon, \Omega, f)$ are both continuous from $\left(-\frac{1}{2}, \frac{1}{2}\right)\times \mathbb{R} \times V^r$ to $Y^{k-1}$, where $V^r$ is a small neighborhood of origin in $X^{k}_{\log}$ $(\alpha=1)$, or $X^{k+\alpha-1}$ $(1<\alpha<2)$. Moreover, we will prove $\partial_fG^\alpha(0, \Omega, 0)$ is an isomorphism from $X^{k}_{\log}$ $(\alpha=1)$, or $X^{k+\alpha-1}$ $(1<\alpha<2)$ to $Y^{k-1}_0$. The key point is to adjust the angular velocity $\Omega=\Omega_\alpha(\varepsilon,f)$ such that the range of $G^\alpha(\varepsilon, \Omega, f)$ is in $Y^{k-1}_0$. Then we can apply  implicit function theorem near $(\varepsilon,\Omega,f)=(0,\Omega_\alpha,0)$ to obtain the existence of co-rotating global solutions.

\subsection{Existence for the case $\alpha=1$}

We first give an alternative characterization for the space $X^{k}_{\log}$.
\begin{lemma}\label{lem2-1}
	For $g(x)=\sum\limits_{j=2}^{\infty}a_j\cos(jx)$, $g\in X^{k}_{\log}$ if and only if
	\begin{equation*}
		g\in X^k, \ \ \ \sum\limits_{j=2}^{\infty}|a_j|^2j^{2k}(1+\ln j)^2<\infty.
	\end{equation*}
\end{lemma}
\begin{proof}
	By Lemma \ref{A-1} in Appendix and Plancherel's identity, it holds
	\begin{equation*}
		\left\|\int_0^{2\pi}\frac{\partial^kg(x-y)-\partial^kg(x)}{|\sin(\frac{y}{2})|}dy\right\|_{L^2}^2=\sum\limits_{j=2}^{\infty}(1+O(\ln j))^2|\widehat{\partial^kg}_j|^2=\sum\limits_{j=2}^{\infty}|a_j|^2j^{2k}(1+O(\ln j))^2,
	\end{equation*}
   from which the equivalence of two characterizations for the space $X^{k}_{\text{log}}$ is obvious.
\end{proof}

Denote $V^r$ as the open neighborhood of zero in $X^{k}_{\text{log}}$
\begin{equation*}
	V^r:=\left\{g\in X^{k}_{\log}: \ \|g\|_{X^{k}_{\log}}<r\right\}
\end{equation*}
with $0<r<1$ and $k\ge 3$. In the following two lemmas, we investigate the regularity of $G^1(\varepsilon, \Omega, f)$.
\begin{lemma}\label{lem2-2}
	$G^1(\varepsilon, \Omega, f): \left(-\frac{1}{2}, \frac{1}{2}\right)\times \mathbb{R} \times V^r \rightarrow Y^{k-1}$ is continuous.
\end{lemma}
\begin{proof}
	By \eqref{2-1} it is easy to show that $G_1: \left(-\frac{1}{2}, \frac{1}{2}\right)\times \mathbb{R} \times V^r \rightarrow Y^{k-1}$ is continuous. We can rewrite
	$G_1$ as 
	\begin{equation}\label{2-4}
		G_1=\Omega \left(d\sin(x)+\varepsilon|\varepsilon|\mathcal{R}_1(\varepsilon,f)\right),
	\end{equation}
    where $\mathcal{R}_1(\varepsilon,f): \left(-\frac{1}{2}, \frac{1}{2}\right)\times \mathbb{R} \times V^r \rightarrow Y^{k-1}$ is continuous. 
	
	To consider $G_2$ given in \eqref{2-2}, we will first prove the range of $G_2$ is in $Y^{k-1}$. Notice that $f(x)\in V^r$ is an even function, and $f'(x)$ is odd. By changing $y$ to $-y$ in $G_2$, we can deduce that $G_2(\varepsilon,f)$ is odd.  
	
	Since $R(x)=1+\varepsilon|\varepsilon| f(x)$, the possible singularity for $\varepsilon=0$ may occur only when we take zeroth derivative of $G_{21}$. To prove this case can not happen, we will resort to Taylor formula:
	\begin{equation}\label{2-5}
		\frac{1}{(A+B)^\lambda}=\frac{1}{A^\lambda}-\lambda\int_0^1\frac{B}{(A+tB)^{1+\lambda}}dt
	\end{equation}
	with
	$$A:=4\sin^2\left(\frac{x-y}{2}\right),$$ 
	and $$B:=\varepsilon|\varepsilon|(f^2(x)-f^2(y))+\sin^2\left(\frac{x-y}{2}\right)\left(4f(x)+4f(y)+\varepsilon|\varepsilon| f(x)f(y)\right).$$
	Then from \eqref{2-2}, we use the fact $\sin(\cdot)$ is odd to obtain 
	\begin{equation}\label{2-6}
		\begin{split}
			G_{21}&=\frac{1}{|\varepsilon|\varepsilon}\int\!\!\!\!\!\!\!\!\!\; {}-{} \frac{\sin(x-y)dy}{\left(A+\varepsilon|\varepsilon| B+O(\varepsilon^4)  \right)^{\frac{1}{2}}}+\int\!\!\!\!\!\!\!\!\!\; {}-{} \frac{f(y)\sin(x-y)dy}{\left(A+\varepsilon|\varepsilon| B+O(\varepsilon^4)   \right)^{\frac{1}{2}}}\\
			&=\frac{1}{\varepsilon^{1+\alpha}}\int\!\!\!\!\!\!\!\!\!\; {}-{} \frac{\sin(x-y)dy}{ A^{\frac{1}{2}}}-\frac{1}{2}\int\!\!\!\!\!\!\!\!\!\; {}-{} \int_0^1 \frac{B\sin(x-y)dt dy}{\left(A+t\varepsilon|\varepsilon| B+O(\varepsilon^4) \right)^{\frac{3}{2}}}+\int\!\!\!\!\!\!\!\!\!\; {}-{} \frac{f(y)\sin(x-y)dy}{\left(A+\varepsilon|\varepsilon|B +O(\varepsilon^4)  \right)^{\frac{1}{2}}}\\
			&=-\frac{1}{2}\int\!\!\!\!\!\!\!\!\!\; {}-{} \int_0^1 \frac{B\sin(x-y)dy}{A^{\frac{3}{2}}}+\frac{3\varepsilon|\varepsilon|}{4}\int\!\!\!\!\!\!\!\!\!\; {}-{}\int_0^1 \int_0^1 \frac{t B\sin(x-y)d\tau dt dy}{\left( A + \tau t\varepsilon|\varepsilon| B +O(\varepsilon^4)  \right)^{\frac{5}{2}}}\\
			&\ \ \ \ +\int\!\!\!\!\!\!\!\!\!\; {}-{} \frac{f(y)\sin(x-y)dy}{A^{\frac{1}{2}}}-\frac{\varepsilon|\varepsilon|}{2}\int\!\!\!\!\!\!\!\!\!\; {}-{} \int_0^1 \frac{Bf(y)\sin(x-y)dt dy}{\left(A+t\varepsilon|\varepsilon| B+O(\varepsilon^4) \right)^{\frac{3}{2}}}\\
			&=\frac{1}{4}\int\!\!\!\!\!\!\!\!\!\; {}-{} \frac{f(y)\sin(x-y)dy}{\left|\sin(\frac{x-y}{2})\right|}+\varepsilon|\varepsilon|\mathcal{R}_{11}(\varepsilon,f),
		\end{split}
	\end{equation}
    where $\mathcal{R}_{11}(\varepsilon,f)$ is not singular with respect to $\varepsilon$.
    
    Now, we take $\partial^{k-1}$ derivatives of $G_2$ with respect to $x$. Since $G_{21}$ is less singular than the other three terms, we will begin with $G_{22}$:
    \begin{equation*}
    	\begin{split}
    		\partial^{k-1}G_{22}&=\int\!\!\!\!\!\!\!\!\!\; {}-{}\frac{(\partial^kf(y)-\partial^kf(x))\cos(x-y)dy}{\left( |\varepsilon|^4\left(f(x)-f(y)\right)^2+4(1+\varepsilon|\varepsilon| f(x))(1+\varepsilon|\varepsilon| f(y))\sin^2\left(\frac{x-y}{2}\right)\right)^{\frac{1}{2}}}\\	
    		& \ \ \ \ -\int\!\!\!\!\!\!\!\!\!\; {}-{}\frac{\cos(x-y)}{\left( |\varepsilon|^4\left(f(x)-f(y)\right)^2+4(1+\varepsilon|\varepsilon| f(x))(1+\varepsilon|\varepsilon| f(y))\sin^2\left(\frac{x-y}{2}\right)\right)^{\frac{3}{2}}}\\
    		& \ \ \ \ \times\left((f(x)-f(y))(f'(x)-f'(y))+2(f(x)f'(y)+f'(x)f(y))\sin^2(\frac{x-y}{2})\right)\\
    		& \ \ \ \ \times(\partial^{k-1}f(y)-\partial^{k-1}f(x))dy +l.o.t,
    	\end{split}
    \end{equation*}
    where $l.o.t$ denotes the lower order terms. Since $f(x)\in X^k_{\log}$ and $k\ge3$, we have $\|\partial^if\|_{L^\infty}\le C \|f\|_{X^k_{\log}}<\infty$ for $i=0,1,2$. By H\"older inequality and mean value theorem, we can conclude that
    \begin{equation*}
    	\begin{split}
    	\left\|\partial^{k-1}G_{22}\right\|_{L^2}&\le C\left\|\int\!\!\!\!\!\!\!\!\!\; {}-{}\frac{\partial^kf(x)-\partial^kf(y)}{|\sin(\frac{x-y}{2})|}dy\right\|_{L^2}+C\left\|\int\!\!\!\!\!\!\!\!\!\; {}-{}\frac{\partial^{k-1}f(x)-\partial^{k-1}f(y)}{|\sin(\frac{x-y}{2})|}dy\right\|_{L^2}\\
    	&\le C \|f\|_{X^k_{\log}}+C\|f\|_{X^{k-1}_{\log}}<\infty
    	\end{split}
    \end{equation*}
    Notice that the bound for $\left\|\partial^{k-1}G_{23}\right\|_{L^2}$ is easier to obtain than $\left\|\partial^{k-1}G_{22}\right\|_{L^2}$. Actually, $G_{23}$ is less singular than $G_{22}$. So we turn to focus on $G_{24}$, and take $\partial^{k-1}$ derivatives to obtain
    \begin{equation*}
    	\begin{split}
    		\partial^{k-1}G_{24}&=\frac{-\varepsilon|\varepsilon|\partial^{k-1}f(x)}{(1+\varepsilon|\varepsilon|f(x))^2}\int\!\!\!\!\!\!\!\!\!\; {}-{}\frac{f'(x)f'(y)\sin(x-y)dy}{\left(|\varepsilon|^4\left(f(x)-f(y)\right)^2+4(1+\varepsilon|\varepsilon| f(x))(1+\varepsilon|\varepsilon| f(y))\sin^2\left(\frac{x-y}{2}\right)\right)^{\frac{1}{2}}}\\	
    		& \ \ \ \ +\frac{\varepsilon|\varepsilon|}{1+\varepsilon|\varepsilon|f(x)}\int\!\!\!\!\!\!\!\!\!\; {}-{}\frac{\sin(x-y)}{\left( |\varepsilon|^4\left(f(x)-f(y)\right)^2+4(1+\varepsilon|\varepsilon| f(x))(1+\varepsilon|\varepsilon| f(y))\sin^2\left(\frac{x-y}{2}\right)\right)^{\frac{1}{2}}}\\   	
    		& \ \ \ \ \times(f'(x)\partial^kf(y)+\partial^kf(x)f'(y))dy\\
    		& \ \ \ \ -	\frac{\varepsilon|\varepsilon|}{1+\varepsilon|\varepsilon|f(x)}\int\!\!\!\!\!\!\!\!\!\; {}-{}\frac{f'(x)f'(y)\sin(x-y)}{\left( |\varepsilon|^4\left(f(x)-f(y)\right)^2+4(1+\varepsilon|\varepsilon| f(x))(1+\varepsilon|\varepsilon| f(y))\sin^2\left(\frac{x-y}{2}\right)\right)^{\frac{3}{2}}}\\
    		& \ \ \ \ \times \left((f(x)-f(y))(f'(x)-f'(y))+2(f(x)f'(y)+f'(x)f(y))\sin^2(\frac{x-y}{2})\right)\\
    		& \ \ \ \ \times(\partial^{k-1}f(y)-\partial^{k-1}f(x))dy +l.o.t,
    	\end{split}
    \end{equation*}
    from which one can deduce
    \begin{equation*}
    	\begin{split}
    		\left\|\partial^{k-1}G_{24}\right\|_{L^2}&\le C\varepsilon|\varepsilon|\left(\|f'\|_{L^\infty}^2\|\partial^{k-1}f\|_{L^2}+\|f'\|_{L^\infty}\|\partial^kf\|_{L^2}+\|f'\|_{L^\infty}^4\|\partial^kf\|_{L^2}\right)\\
    		&\le C \varepsilon|\varepsilon|\|f\|_{X^k_{\log}}<\infty
    	\end{split}
    \end{equation*}
    As a result, the range of $G_2$ is in $Y^{k-1}$.
     
    To prove the continuity of $G_2$, we will also deal with the most singular term $G_{22}$, and use following notations: For a general function $g$, we let
    $$\Delta g=g(x)-g(y), \ \ \ g=g(x), \ \ \ \tilde g=g(y),$$
    and
    $$D_\alpha(g)=\varepsilon^{2+2\alpha}\Delta g^2+4(1+\varepsilon|\varepsilon|^\alpha g)(1+\varepsilon|\varepsilon|^\alpha\tilde g)\sin^2(\frac{x-y}{2}).$$
    Then for $f_1,f_2\in V^r$, it holds
    \begin{equation*}
    	\begin{split}
    		G_{22}(\varepsilon, f_1)&-G_{22}(\varepsilon, f_2)=\int\!\!\!\!\!\!\!\!\!\; {}-{}\frac{(\Delta f'_1-\Delta f'_2)\cos(x-y)dy}{D_1(f_1)^\frac{1}{2}}\\
    		&+\left(\int\!\!\!\!\!\!\!\!\!\; {}-{}\frac{\Delta f'_2\cos(x-y)dy}{D_1(f_1)^\frac{1}{2}}-\int\!\!\!\!\!\!\!\!\!\; {}-{}\frac{\Delta f'_2\cos(x-y)dy}{D_1(f_2)^\frac{1}{2}}\right)\\
    		&=I_1+I_2.
    	\end{split}
    \end{equation*}
    It is easy to prove that $\|I_1\|_{Y^{k-1}}\le C\|f_1-f_2\|_{X^k_{\text{log}}}$. To consider $I_2$, we apply once mean value theorem to obtain
    \begin{equation}\label{2-7}
    	\begin{split}
    		&\frac{1}{D_\alpha(f_1)^\frac{\alpha}{2}}-\frac{1}{D_\alpha(f_2)^\frac{\alpha}{2}}=\frac{\alpha}{2}\frac{D_\alpha(f_2)^\frac{\alpha}{2}-D^\alpha(f_1)^\frac{\alpha}{2}}{D_\alpha(\delta_{x,y}f_1+(1-\delta_{x,y})f_2)^{1-\frac{\alpha}{2}}D_\alpha(f_1)^\frac{\alpha}{2}D_\alpha(f_2)^\frac{\alpha}{2}}\\
    		&=\frac{\alpha}{2}\frac{|\varepsilon|^{2+2\alpha}(\Delta f_2^2-\Delta f_1^2)+4\varepsilon|\varepsilon|^\alpha((f_2-f_1)(1+\varepsilon|\varepsilon|^\alpha\tilde f_2)+(\tilde f_2-\tilde f_1)(1+\varepsilon|\varepsilon|^\alpha f_1)\sin^2(\frac{x-y}{2})}{D_\alpha(\delta_{x,y}f_1+(1-\delta_{x,y})f_2)^{1-\frac{\alpha}{2}}D_\alpha(f_1)^\frac{\alpha}{2}D_\alpha(f_2)^\frac{\alpha}{2}}
    	\end{split}
    \end{equation}
    for some $\delta_{x,y}\in (0,1)$. Notice that for $g\in X^k_{\log}$. It holds $D_1(g)\sim \sin^2(\frac{x-y}{2})\sim |x-y|^2/4$ as $|x-y|\to 0$. Since $\alpha=1$, one has
    \begin{equation*}
    	\begin{split}
    	\partial^{k-1}I_2\sim C&\int\!\!\!\!\!\!\!\!\!\; {}-{}\frac{\partial^{k-1}f_2(x)-\partial^{k-1}f_2(y)}{|\sin(\frac{x-y}{2})|^\alpha}\times \left(\frac{|\varepsilon|^4(\Delta f_2^2-\Delta f_1^2)}{|x-y|^2}+4\varepsilon|\varepsilon|(f_2-f_1+\tilde f_2-\tilde f_1)\right)dy\\
    	&+l.o.t,
    	\end{split}
    \end{equation*}
    and we deduce that $\|I_2\|_{Y^{k-1}}\le C\|f_1-f_2\|_{X^k_{\log}}$. So we have proven $G_2(\varepsilon, f): \left(-\frac{1}{2}, \frac{1}{2}\right)\times V^r \rightarrow Y^{k-1}$ is continuous. In \eqref{2-6}, we have already collect the main terms of $G_{21}$. For a future use, we also apply Taylor formula \eqref{2-5} on $G_{22}$ and $G_{23}$ to obtain
    \begin{equation}\label{2-8}
    	\begin{split}
    		G_2&=\frac{1}{4}\int\!\!\!\!\!\!\!\!\!\; {}-{} \frac{f(y)\sin(x-y)dy}{\sin(\frac{x-y}{2})}+\frac{1}{2}\int\!\!\!\!\!\!\!\!\!\; {}-{} \frac{(f'(y)-f'(x))\cos(x-y)dy}{\sin(\frac{x-y}{2})}+\varepsilon|\varepsilon|\mathcal{R}_{2}(\varepsilon,f)\\
    		&=\frac{1}{4}\int\!\!\!\!\!\!\!\!\!\; {}-{} \frac{f(x-y)\sin(y)dy}{\sin(\frac{y}{2})}-\frac{1}{2}\int\!\!\!\!\!\!\!\!\!\; {}-{} \frac{(f'(x)-f'(x-y))\cos(y)dy}{\sin(\frac{y}{2})}+\varepsilon|\varepsilon|\mathcal{R}_{2}(\varepsilon,f),
    	\end{split}
    \end{equation}
    where $\mathcal{R}_2(\varepsilon, f): \left(-\frac{1}{2}, \frac{1}{2}\right)\times V^r \rightarrow Y^{k-1}$ is continuous by previous discussion.
    
    Then we move to consider $G_3$. To eliminate the possible singularity at $\varepsilon=0$, we will apply Taylor formula \eqref{2-5} on $G_{31}$ with
    $$A_i=\left((-1+\cos(\frac{2\pi i}{m}))^2 +\sin^2(\frac{2\pi i}{m})\right)d^2,$$
    $$B_i=2d(-1+\cos(\frac{2\pi i}{m}))(\cos(x)-\cos(y+\frac{2\pi i}{m}))+2d\sin(\frac{2\pi i}{m})(\sin(x)-\sin(x+\frac{2\pi i}{m})).$$
    Since $\sin(\cdot)$ is an odd function, from \eqref{2-3} we have 
    \begin{equation*}
    	\begin{split}
    		G_{31}
    		&=\sum_{i=1}^{m-1} \frac{1}{\varepsilon} \int\!\!\!\!\!\!\!\!\!\; {}-{} \frac{\sin(x-y-\frac{2\pi i}{m})dy}{\left( A_i+\varepsilon B_i+O(\varepsilon^2) \right)^{\frac{1}{2}}}+\sum\limits_{i=1}^{m-1}\int\!\!\!\!\!\!\!\!\!\; {}-{} \frac{|\varepsilon|f(y)\sin(x-y-\frac{2\pi i}{m})dy}{\left| \left(\boldsymbol{z}(x)-(d,0)\right)-Q_{\frac{2\pi i}{m}}\left(\boldsymbol{z}(y)-(d,0)\right)\right|}\\
    		&=\sum_{i=1}^{m-1} \frac{1}{\varepsilon} \int\!\!\!\!\!\!\!\!\!\; {}-{} \frac{\sin(x-y-\frac{2\pi i}{m})dy}{ A_i^{\frac{1}{2}}}-\sum_{i=1}^{m-1} \frac{1}{2} \int\!\!\!\!\!\!\!\!\!\; {}-{}\int_0^1 \frac{(B_i+O(\varepsilon))\sin(x-y-\frac{2\pi k}{m})dtdy}{\left( A_i+\varepsilon tB_i+O(\varepsilon^2) \right)^{\frac{3}{2}}}\\
    		& \ \ \ \ +\sum\limits_{i=1}^{m-1}\int\!\!\!\!\!\!\!\!\!\; {}-{} \frac{|\varepsilon|f(y)\sin(x-y-\frac{2\pi i}{m})dy}{\left| \left(\boldsymbol{z}(x)-(d,0)\right)-Q_{\frac{2\pi i}{m}}\left(\boldsymbol{z}(y)-(d,0)\right)\right|}\\
    		&=-\sum_{i=1}^{m-1} \frac{1}{2} \int\!\!\!\!\!\!\!\!\!\; {}-{} \frac{B_i\sin(x-y-\frac{2\pi i}{m})dy}{A_i^{\frac{3}{2}}}+\varepsilon\mathcal{R}_{31}(\varepsilon,f)\\
    		&=-\sum_{i=1}^{m-1} \frac{(-1+\cos(\frac{2\pi i}{m}))\sin (x)}{2\left((-1+\cos(\frac{2\pi i}{m}))^2 +\sin^2(\frac{2\pi i}{m})\right)^{\frac{3}{2}}d^2}+\varepsilon\mathcal{R}_{31}(\varepsilon,f),
    	\end{split}
    \end{equation*}
    where $\mathcal{R}_{31}(\varepsilon,f)$ is not singular with respect to $\varepsilon$.
    
    Notice that for $\boldsymbol{z}\in \partial D^\varepsilon_0$ and $i\ge 1$, $\left| \left(\boldsymbol{z}(x)-(d,0)\right)-Q_{\frac{2\pi i}{m}}\left(\boldsymbol{z}(y)-(d,0)\right)\right|$ has positive lower bounds. So $G_3$ is less singular than $G_2$, and we can also deduce $G_3(\varepsilon, f): \left(-\frac{1}{2}, \frac{1}{2}\right)\times V^r \rightarrow Y^{k-1}$ is continuous. Moreover, it holds
    \begin{equation}\label{2-9}
    	G_3=-\sum_{i=1}^{m-1} \frac{(-1+\cos(\frac{2\pi i}{m}))\sin (x)}{2\left((-1+\cos(\frac{2\pi i}{m}))^2 +\sin^2(\frac{2\pi i}{m})\right)^{\frac{3}{2}}d^2}+\varepsilon\mathcal{R}_{3}(\varepsilon,f),
    \end{equation}
    where $\mathcal{R}_3(\varepsilon, f): \left(-\frac{1}{2}, \frac{1}{2}\right)\times V^r \rightarrow Y^{k-1}$ is continuous. 
    
    The proof is complete by concluding all the facts above.   
\end{proof}

For $(\varepsilon,\Omega,f)\in \left(-\frac{1}{2}, \frac{1}{2}\right)\times \mathbb{R} \times V^r$ and $h\in X^{k}_{\log}$, let
\begin{equation*}
	\partial_fG^1(\varepsilon, \Omega, f)h:=\lim\limits_{t\to0}\frac{1}{t}\left(G^1(\varepsilon, \Omega, f+th)-G^1(\varepsilon, \Omega, f)\right)
\end{equation*}
be the Gateaux derivative of $G^1(\varepsilon, \Omega, f)$. We have following lemma:
\begin{lemma}\label{lem2-3}
	For each $(\varepsilon,\Omega,f)\in \left(-\frac{1}{2}, \frac{1}{2}\right)\times \mathbb{R} \times V^r$, $\partial_fG^1(\varepsilon, \Omega, f)h: X^{k}_{\log}\to Y^{k-1}$ is continuous.
\end{lemma}
\begin{proof}
	From \eqref{2-1}, it is easy to see that
	\begin{equation}\label{2-10}
	\partial_f G_1(\varepsilon,\Omega, f)h=\Omega\varepsilon|\varepsilon|\partial_f\mathcal{R}_1(\varepsilon,f)h
	\end{equation}
	is continuous, where $\mathcal{R}_1(\varepsilon,f)$ is given in \eqref{2-4}.
	
	Next, we claim $\partial_f G_2(\varepsilon, f)h=F_1+F_2+F_3+F_4$ is continuous, where
	\begin{equation}\label{2-11}
		\begin{split}		
			&F_1=\int\!\!\!\!\!\!\!\!\!\; {}-{} \frac{h(y)\sin(x-y)dy}{\left( |\varepsilon|^4\left(f(x)-f(y)\right)^2+4(1+\varepsilon|\varepsilon| f(x))(1+\varepsilon|\varepsilon| f(y))\sin^2\left(\frac{x-y}{2}\right)\right)^{\frac{1}{2}}}\\
			& \ \ \ \ -\frac{1}{2}\int\!\!\!\!\!\!\!\!\!\; {}-{} \frac{(1+\varepsilon|\varepsilon|f(y))\sin(x-y)}{\left( |\varepsilon|^4\left(f(x)-f(y)\right)^2+4(1+\varepsilon|\varepsilon| f(x))(1+\varepsilon|\varepsilon| f(y))\sin^2\left(\frac{x-y}{2}\right)\right)^{\frac{3}{2}}}\\
			& \ \ \ \ \times\bigg(\varepsilon|\varepsilon|(2(f(x)-f(y))(h(x)-h(y))\\
			& \ \ \ \ \ \ \ \ +4(h(x)(1+\varepsilon|\varepsilon|f(y))+h(y)(1+\varepsilon|\varepsilon|f(x)))\sin^2(\frac{x-y}{2})\bigg)dy
		\end{split}
	\end{equation}
    \begin{equation}\label{2-12}
    	\begin{split}
    		F_2&=\int\!\!\!\!\!\!\!\!\!\; {}-{} \frac{(h'(y)-h'(x))\cos(x-y)dy}{\left( |\varepsilon|^4\left(f(x)-f(y)\right)^2+4(1+\varepsilon|\varepsilon| f(x))(1+\varepsilon|\varepsilon| f(y))\sin^2\left(\frac{x-y}{2}\right)\right)^{\frac{1}{2}}}\\
    		& \ \ \ \ -\frac{\varepsilon|\varepsilon|}{2}\int\!\!\!\!\!\!\!\!\!\; {}-{} \frac{(f'(y)-f'(x))\cos(x-y)}{\left( |\varepsilon|^4\left(f(x)-f(y)\right)^2+4(1+\varepsilon|\varepsilon| f(x))(1+\varepsilon|\varepsilon| f(y))\sin^2\left(\frac{x-y}{2}\right)\right)^{\frac{3}{2}}}\\
    		& \ \ \ \ \times\bigg(\varepsilon|\varepsilon|(2(f(x)-f(y))(h(x)-h(y))\\
    		& \ \ \ \ \ \ \ \ +4(h(x)(1+\varepsilon|\varepsilon|f(y))+h(y)(1+\varepsilon|\varepsilon|f(x)))\sin^2(\frac{x-y}{2})\bigg)dy
    	\end{split}
    \end{equation}
    \begin{equation}\label{2-13}
    	\begin{split}
    		F_3&=\frac{\varepsilon|\varepsilon|h'(x)}{1+\varepsilon|\varepsilon|f(x)}\int\!\!\!\!\!\!\!\!\!\; {}-{} \frac{(f(x)-f(y))\cos(x-y)dy}{\left( |\varepsilon|^4\left(f(x)-f(y)\right)^2+4(1+\varepsilon|\varepsilon| f(x))(1+\varepsilon|\varepsilon| f(y))\sin^2\left(\frac{x-y}{2}\right)\right)^{\frac{1}{2}}}\\
    		& \ \ \ \ -\frac{\varepsilon|\varepsilon|f'(x)h(x)}{(1+\varepsilon|\varepsilon|f(x))^2}\int\!\!\!\!\!\!\!\!\!\; {}-{} \frac{(f(x)-f(y))\cos(x-y)dy}{\left( |\varepsilon|^4\left(f(x)-f(y)\right)^2+4(1+\varepsilon|\varepsilon| f(x))(1+\varepsilon|\varepsilon| f(y))\sin^2\left(\frac{x-y}{2}\right)\right)^{\frac{1}{2}}}\\
    		& \ \ \ \ +\frac{\varepsilon|\varepsilon|f'(x)}{1+\varepsilon|\varepsilon|f(x)}\int\!\!\!\!\!\!\!\!\!\; {}-{} \frac{(h(x)-h(y))\cos(x-y)dy}{\left( |\varepsilon|^4\left(f(x)-f(y)\right)^2+4(1+\varepsilon|\varepsilon| f(x))(1+\varepsilon|\varepsilon| f(y))\sin^2\left(\frac{x-y}{2}\right)\right)^{\frac{1}{2}}}\\
    		& \ \ \ \ -\frac{\varepsilon|\varepsilon|f'(x)}{1+\varepsilon|\varepsilon|f(x)}\int\!\!\!\!\!\!\!\!\!\; {}-{} \frac{(f(x)-f(y))\cos(x-y)}{\left( |\varepsilon|^4\left(f(x)-f(y)\right)^2+4(1+\varepsilon|\varepsilon| f(x))(1+\varepsilon|\varepsilon| f(y))\sin^2\left(\frac{x-y}{2}\right)\right)^{\frac{3}{2}}}\\
    		& \ \ \ \ \times\bigg(\varepsilon|\varepsilon|(2(f(x)-f(y))(h(x)-h(y))\\
    		& \ \ \ \ \ \ \ \ +4(h(x)(1+\varepsilon|\varepsilon|f(y))+h(y)(1+\varepsilon|\varepsilon|f(x)))\sin^2(\frac{x-y}{2})\bigg)dy		
    	\end{split}
    \end{equation}
    \begin{equation}\label{2-14}
    	\begin{split}
    		&F_4=\frac{|\varepsilon|^4}{1+\varepsilon|\varepsilon|f(x)}\int\!\!\!\!\!\!\!\!\!\; {}-{} \frac{h'(y)f'(x)\sin(x-y)dy}{\left( |\varepsilon|^4\left(f(x)-f(y)\right)^2+4(1+\varepsilon|\varepsilon| f(x))(1+\varepsilon|\varepsilon| f(y))\sin^2\left(\frac{x-y}{2}\right)\right)^{\frac{1}{2}}}\\
    		& \ \ \ \ -\frac{|\varepsilon|^4}{2(1+\varepsilon|\varepsilon|f(x))}\int\!\!\!\!\!\!\!\!\!\; {}-{} \frac{f'(x)f'(y)\sin(x-y)}{\left( |\varepsilon|^4\left(f(x)-f(y)\right)^2+4(1+\varepsilon|\varepsilon| f(x))(1+\varepsilon|\varepsilon| f(y))\sin^2\left(\frac{x-y}{2}\right)\right)^{\frac{3}{2}}}\\
    		& \ \ \ \ \times\bigg(\varepsilon|\varepsilon|(2(f(x)-f(y))(h(x)-h(y))\\
    		& \ \ \ \ \ \ \ \ +4(h(x)(1+\varepsilon|\varepsilon|f(y))+h(y)(1+\varepsilon|\varepsilon|f(x)))\sin^2(\frac{x-y}{2})\bigg)dy	
    	\end{split}
    \end{equation}

    To this aim, the first step is to show
    $$\lim\limits_{t\to0}\left\|\frac{G_{2i}(\varepsilon, f+th)-G_{2i}(\varepsilon, f)}{t}-F_i(\varepsilon, f,h)\right\|_{Y^{k-1}}\to 0$$
	for $i=1,2,3,4$. For simplicity, we only consider the most singular case $i=2$ and use the notations given in Lemma \ref{2-2}. It holds
	\begin{equation*}
		\begin{split}
			&\frac{G_{22}(\varepsilon, f+th)-G_{22}(\varepsilon, f)}{t}-F_i(\varepsilon, f,h)\\
			&=\frac{1}{t}\int\!\!\!\!\!\!\!\!\!\; {}-{}(f'(x)-f'(y))\cos(x-y)\bigg(\frac{1}{D_1(f+th)}-\frac{1}{D_1(f)}+t\frac{\Delta f\Delta h+2(\tilde fh+h\tilde f)\sin^2(\frac{x-y}{2})}{D_1(f)}\bigg)dy\\
			& \ \ \ \ +\int\!\!\!\!\!\!\!\!\!\; {}-{}(h'(x)-h'(y))\cos(x-y)\bigg(\frac{1}{D_1(f+th)}-\frac{1}{D_1(f)}\bigg)dy\\
			&=F_{21}+F_{22}
		\end{split}
	\end{equation*}
    By taking $\partial^{k-1}$ derivatives of $F_{21}$, we find
    \begin{equation*}
    	\begin{split}
    	\partial^{k-1}F_{21}&=\frac{1}{t}\int\!\!\!\!\!\!\!\!\!\; {}-{}\bigg(\frac{1}{D_1(f+th)}-\frac{1}{D_1(f)}+t\frac{\Delta f\Delta h+2(\tilde fh+h\tilde f)\sin^2(\frac{x-y}{2})}{D_1(f)}\bigg)\\
    	& \ \ \ \ \times(\partial^kf(x)-\partial^kf(y))\cos(x-y)dy+l.o.t.
    	\end{split}
    \end{equation*}
    Using mean value theorem, we derive
    \begin{equation*}
    	\frac{1}{D_1(f+th)}-\frac{1}{D_1(f)}+t\frac{\Delta f\Delta h+2(\tilde fh+h\tilde f)\sin^2(\frac{x-y}{2})}{D_1(f)}\sim \frac{Ct^2}{|\sin(\frac{x-y}{2})|}\zeta(\varepsilon,f,h)
    \end{equation*}
    where $\|\zeta(\varepsilon,f,h)\|_{L^\infty}<\infty$. It follows that 
    \begin{equation*}
    	\|F_{21}\|_{Y^{k-1}}\le Ct\left\|\int\!\!\!\!\!\!\!\!\!\; {}-{}\frac{\partial^kf(x)-\partial^kf(y)}{|\sin(\frac{x-y}{2})|}dy\right\|_{L^2}\le Ct\|f\|_{X^k_{\log}}.
    \end{equation*}
    By \eqref{2-7}, we can also prove $\|F_{22}\|_{Y^{k-1}}\le Ct\|f\|_{X^k_{\log}}$. So the first step is finished by letting $t\to 0$. The second step is to prove the continuity of $\partial_f G_2(\varepsilon, f)h$, which relys on \eqref{2-7}. Since there is no other new idea than the proof of continuity for $G^1(\varepsilon, \Omega, f)$, we omit it. 
    
    Using a similar method as above, we deduce that
    \begin{equation}\label{2-15}
    	\partial_f G_3(\varepsilon, f)h=|\varepsilon|\partial_f\mathcal{R}_3(\varepsilon,f)h
    \end{equation}
    is continuous, where $\mathcal{R}_3(\varepsilon,f)$ is the same term as in \eqref{2-9}. This completes the proof of Lemma \ref{lem2-3}.
\end{proof}

From \eqref{2-10}-\eqref{2-15}, by letting $\varepsilon=0$ and $f\equiv 0$, one has
\begin{equation*}
	\partial_fG^1(0,\Omega,0)h=\frac{1}{2}\int\!\!\!\!\!\!\!\!\!\; {}-{} \frac{h(x-y)\sin(y)dy}{\left(4\sin^2(\frac{y}{2})\right)^{\frac{1}{2}}}-\int\!\!\!\!\!\!\!\!\!\; {}-{} \frac{(h'(x)-h'(x-y))\cos(y)dy}{\left(4\sin^2(\frac{y}{2})\right)^{\frac{1}{2}}}
\end{equation*}
The following lemma claims that the linearization of $G^1(\varepsilon,\Omega,f)$ at $(0,\Omega,0)$ is an isomorphism, which is crutial for our purpose to apply implicit function theorem.
\begin{lemma}\label{lem2-4}
	Let $h=\sum\limits_{j=2}^\infty a_j\cos(jx)$ be in $X^k_{\log}$. Then it holds
	\begin{equation*}
		\partial_fG^1(0,\Omega,0)h=\sum\limits_{j=2}^\infty a_j \gamma_j j\sin(jx) 
	\end{equation*}
	with
	\begin{equation*}
		\gamma_j=\frac{2}{\pi}\sum\limits_{i=1}^j\frac{1}{2i-1}.
	\end{equation*}
    Moreover, for each $\Omega\in\mathbb R$, $\partial_fG^1(0,\Omega,0)h:X^k_{\log}\to Y_0^{k-1}$ is an isomorphism.
\end{lemma}
\begin{proof}
	We first calculate the formulation of $\partial_fG^1(0,\Omega,0)h$. For the future use, we will deal with the general case $1\le\alpha<2$, namely, the formula
	\begin{equation*}
		C_\alpha(1-\frac{\alpha}{2})\int\!\!\!\!\!\!\!\!\!\; {}-{} \frac{h(x-y)\sin(y)dy}{\left(4\sin^2(\frac{y}{2})\right)^{\frac{\alpha}{2}}}-C_\alpha\int\!\!\!\!\!\!\!\!\!\; {}-{} \frac{(h'(x)-h'(x-y))\cos(y)dy}{\left(4\sin^2(\frac{y}{2})\right)^{\frac{\alpha}{2}}}.
	\end{equation*}
	Since $h(x)$ can be written as $\sum\limits_{j=2}^\infty a_j\cos(jx)$, we will calculate that $\partial_fG^1(0,\Omega,0)$ acts on $a_j\cos(jx)$ for each $j\ge 2$. Using identities
    $$\cos(\frac{y}{2})\left|\sin(\frac{y}{2})\right|^{1-\alpha}=\frac{2}{2-\alpha}\partial_y\left(\left|\sin(\frac{y}{2})\right|^{1-\alpha}\right),$$
    \begin{equation}\label{2-16}
    \int_0^\pi(\sin(y))^{2-\alpha}e^{jy\text i}dy=\frac{\pi e^{j\pi\text i}\Gamma(3-\alpha)}{2^{1-\alpha}\Gamma(2+j-\frac{\alpha}{2})\Gamma(2-j-\frac{\alpha}{2})}, \ \ \ \ \forall \, \alpha<3, \ \ \forall \, j\in\mathbb{R},
    \end{equation}
	and integrating by parts, for the first term we derive
	\begin{equation}\label{2-17}
		\begin{split}
			&C_\alpha(1-\frac{\alpha}{2})\int\!\!\!\!\!\!\!\!\!\; {}-{} \frac{a_j\cos(jx-jy)\sin(y)dy}{\left(4\sin^2(\frac{y}{2})\right)^{\frac{\alpha}{2}}}\\
			&=2^{1-\alpha}C_\alpha(1-\frac{\alpha}{2})\int\!\!\!\!\!\!\!\!\!\; {}-{}a_j\cos(jx-jy)\cos(\frac{y}{2})|\sin(\frac{y}{2})|^{1-\alpha}dy\\
			&=2^{1-\alpha}C_\alpha(1-\frac{\alpha}{2})\frac{-2j}{2-\alpha}a_j\int\!\!\!\!\!\!\!\!\!\; {}-{} \sin(jx-jy)|\sin(\frac{y}{2})|^{2-\alpha}dy\\
			&=2^{1-\alpha}C_\alpha(1-\frac{\alpha}{2})\frac{-2j}{2-\alpha}a_j\sin(jx)\int\!\!\!\!\!\!\!\!\!\; {}-{} \cos(jy)|\sin(\frac{y}{2})|^{2-\alpha}dy\\
			&=2^{1-\alpha}C_\alpha(1-\frac{\alpha}{2})\frac{-2j}{2-\alpha}a_j\frac{\pi \cos(j\pi)\Gamma(3-\alpha)}{2^{1-\alpha}\Gamma(2+j-\frac{\alpha}{2})\Gamma(2-j-\frac{\alpha}{2})}\sin(jx)
		\end{split}
	\end{equation}
    When $\alpha=1$, since $\Gamma(x)\Gamma(1-x)=\pi/\sin(\pi x)$, \eqref{2-17} is equal to $$\frac{1}{\pi}j a_j\frac{2}{4j^2-1}\sin(jx).$$
    
    For the second term, it holds
    \begin{equation*}
    	\begin{split}
    		&C_\alpha j a_j\int\!\!\!\!\!\!\!\!\!\; {}-{} \frac{(\sin(jx)-\sin(jx-jy))\cos(y)dy}{\left(4\sin^2(\frac{y}{2})\right)^{\frac{\alpha}{2}}}\\
    		&=2^{-\alpha}C_\alpha j a_j\int\!\!\!\!\!\!\!\!\!\; {}-{} (\sin(jx)-\sin(jx-jy))\left|\sin(\frac{y}{2})\right|^{-\alpha}dy\\
    		& \ \ \ \ -2^{1-\alpha}C_\alpha j a_j\int\!\!\!\!\!\!\!\!\!\; {}-{} (\sin(jx)-\sin(jx-jy))\left|\sin(\frac{y}{2})\right|^{2-\alpha}dy.
    	\end{split}
    \end{equation*}
    According to Lemma \ref{A-1} in the Appendix and identity \eqref{2-16}, the above equality can be rewritten as
    \begin{equation}\label{2-18}
    	\begin{split}
    		& \ \ \ \ C_\alpha j a_j\frac{2\pi \Gamma(1-\alpha)}{\Gamma(\frac{\alpha}{2})\Gamma(1-\frac{\alpha}{2})}\left(\frac{\Gamma(\frac{\alpha}{2})}{\Gamma(1-\frac{\alpha}{2})}-\frac{\Gamma(j+\frac{\alpha}{2})}{\Gamma(1+j-\frac{\alpha}{2})}\right)\sin(jx)\\
    		&-2^{-1}C_\alpha j a_j\frac{2\pi \Gamma(3-\alpha)}{\Gamma(\frac{\alpha}{2}-1)\Gamma(2-\frac{\alpha}{2})}\left(\frac{\Gamma(\frac{\alpha}{2}-1)}{\Gamma(2-\frac{\alpha}{2})}-\frac{\Gamma(j-1+\frac{\alpha}{2})}{\Gamma(2+j-\frac{\alpha}{2})}\right)\sin(jx)
    	\end{split}
    \end{equation} 
    When $\alpha=1$, it is equal to
    $$\frac{1}{\pi}j a_j\sum\limits_{i=1}^{j}\frac{2}{2j-1}\sin(jx)-\frac{1}{\pi}j a_j\left(2+\frac{2}{4j^2-1}\right)\sin(jx).$$
    Hence we derive
    \begin{equation*}
        \partial_fG^1(0,\Omega,0)h=\sum\limits_{j=2}^\infty a_j \gamma_j j\sin(jx), \ \ \ \ \ \gamma_j=\frac{2}{\pi}\sum\limits_{i=1}^j\frac{1}{2i-1}.
    \end{equation*}
   
   Then we have to verify that $\partial_fG^1(0,\Omega,0):X^k_{\log}\to Y_0^{k-1}$ is an isomorphism. By Lemma \ref{A-1}, $\{\gamma_j\}$ is a monotone sequence with posive lower bound. So the kernel of $\partial_fG^1(0,\Omega,0)$ is trivial. On the  other hand, we need to show that for each $p(x)\in Y_0^{k-1}$ with the form $p(x)=\sum\limits_{j=2}^\infty b_j\sin(jx)$, there exists an $h(x) \in X^k_{\log}$ such that
   $\partial_fG^1(0,\Omega,0)h=p$. From above calculation, we see that this $h$ can be given directly by
   \begin{equation*}
   	   h(x)=\sum\limits_{j=2}^\infty b_j\gamma_j^{-1}j^{-1}\cos(jx).
   \end{equation*}
   Since it holds $\gamma_j=O(\ln j)$. By Lemma \ref{A-1}, we have
   $$\|h\|^2_{X^k_{\log}}=\sum\limits_{j=2}^\infty b_j^2\gamma_j^{-2}j^{2k-2}(1+\ln(j))^2\le C\sum\limits_{j=2}^\infty b_j^2j^{2k-2}\left(\frac{1+\ln(j)}{\ln(j)}\right)\le C\|p\|^2_{Y_0^{k-1}}.$$
   As a result, we deduce $h(x) \in X^k_{\text{log}}$ and complete the proof.
\end{proof}

According to \eqref{2-4} \eqref{2-8} \eqref{2-9}, if we let
\begin{equation}\label{2-19}
	\Omega_1^*:=\sum_{i=1}^{m-1} \frac{(-1+\cos(\frac{2\pi i}{m}))}{2\left((-1+\cos(\frac{2\pi i}{m}))^2 +\sin^2(\frac{2\pi i}{m})\right)^{\frac{3}{2}}d^3}
\end{equation}
then it holds $G^1(0,\Omega_1^*,0)=0$. In the next lemma, we are going to adjust the value of angular velocity $\Omega=\Omega_1(\varepsilon,f)$, so that the range of $G^1(\varepsilon,\Omega_1(\varepsilon,f),f)$ is in $Y_0^{k-1}$.
\begin{lemma}\label{lem2-5}
	There exists
	\begin{equation*}
		\Omega_1(\varepsilon, f):=\Omega_1^*+\varepsilon\mathcal{R}_\Omega(\varepsilon,f)
	\end{equation*}
    with $\Omega_1^*$ given in \eqref{2-19} and continuous function $\mathcal{R}_\Omega(\varepsilon,f):X^k_{\log}\to \mathbb{R}$, such that $\tilde G^1(\varepsilon,f):(-\frac{1}{2},\frac{1}{2})\times V^r\to Y_0^{k-1}$ is given by
    \begin{equation*}
    	\tilde G^1(\varepsilon,f):=G^1(\varepsilon,\Omega_1(\varepsilon,f), f).
    \end{equation*}
    Moreover, $\partial_f\mathcal{R}_\Omega(\varepsilon,f)h:X^k_{\log}\to \mathbb{R}$ is continuous.
\end{lemma}
\begin{proof}
	It suffices to find $\Omega_1(\varepsilon, f):(-\frac{1}{2},\frac{1}{2})\times V^r\to\mathbb{R}$, such that the first Fourier coefficient vanishes in $G^1(\varepsilon,\Omega_1(\varepsilon,f), f)$. Notice that from \eqref{2-8} and the proof of Lemma \ref{lem2-4}, the contribution of $G_2$ to the first Fourier coefficient is $\varepsilon|\varepsilon|\tilde{\mathcal{R}}_2(\varepsilon, f)$, with $\tilde{\mathcal{R}}_2$ the contribution of $\mathcal{R}_2$ in \eqref{2-8}. If we combine this fact with \eqref{2-4} and \eqref{2-9}, we deduce that $\Omega_1(\varepsilon, f)$ must satisfy
	\begin{equation*}
	\Omega_1\left(d+\varepsilon|\varepsilon|\tilde{\mathcal{R}}_1(\varepsilon,f)\right)+\varepsilon|\varepsilon|\tilde{\mathcal{R}}_2(\varepsilon, f)-\sum_{i=1}^{m-1} \frac{(-1+\cos(\frac{2\pi i}{m}))}{2\left((-1+\cos(\frac{2\pi i}{m}))^2 +\sin^2(\frac{2\pi i}{m})\right)^{\frac{3}{2}}d^2}+\varepsilon\tilde{\mathcal{R}}_{3}(\varepsilon,f)=0
	\end{equation*}
    with $\tilde{\mathcal{R}}_i(\varepsilon,f)$ $(i=1,2,3)$ the contribution of $\mathcal R_i$ to the first Fourier coefficient.	Direct calculation yields
    \begin{equation*}
    	\Omega_1(\varepsilon, f):=\Omega_1^*+\varepsilon\mathcal{R}_\Omega(\varepsilon,f),
    \end{equation*}
    where $\mathcal{R}_\Omega(\varepsilon,f):X^k_{\log}\to \mathbb{R}$ is some continuous function. Since $\partial_f{\mathcal{R}}_i(\varepsilon,f)$ $(i=1,2,3)$ are continuous from Lemma \ref{lem2-3},  $\partial_f\mathcal{R}_\Omega(\varepsilon,f)h:X^k_{\log}\to \mathbb{R}$ is also continuous. The proof is thus complete.	 
\end{proof}

Now, we are in the position to prove Theorem \ref{thm1}.

{\bf Proof of Theorem \ref{thm1}:}
We first prove $\partial_f\tilde G^1(\varepsilon,f)h: X^k_{\log}\to Y^{k-1}_0$ is an isomorphism. By chain rule, it holds
\begin{equation*}
	\partial_f\tilde G^1(0,0)h=\partial_\Omega G^1(0,\Omega_1^*,0)\partial_f\Omega_1(0,0)h+\partial_f G^1(0,\Omega_1^*,0).
\end{equation*}
From Lemma \ref{lem2-4}, we know $\partial_f\Omega_1(0,0)=0$. Hence $\partial_f\tilde G^1(0,0)h=\partial_f G^1(0,\Omega_1^*,0)$, and we achieve the desired result by Lemma \ref{lem2-4}.

According to Lemma \ref{lem2-1}--Lemma \ref{lem2-5}, we can apply implicit function theorem, and claim that there exists $\varepsilon_0>0$ such that
\begin{equation*}
	\left\{(\varepsilon,f)\in [-\varepsilon_0,\varepsilon_0]\times V^r \ : \ \tilde G_1(\varepsilon,f)=0\right\}
\end{equation*}
is parametrized by one-dimensional curve $\varepsilon\in [-\varepsilon_0,\varepsilon_0]\to (\varepsilon, f_\varepsilon)$. Moreover, we want to show that for each $\varepsilon\in[-\varepsilon_0,\varepsilon_0]\setminus \{0\}$, it always holds $f_\varepsilon\neq 0$. For this purpose, we can apply Taylor formula \eqref{2-5} again on \eqref{2-10} to expand $\mathcal{R}_3(\varepsilon,f)$ as
\begin{equation*}
	\mathcal{R}_3(\varepsilon,f)=-\varepsilon \hat C\sin(2x)+o(\varepsilon)
\end{equation*}
with $\hat C$ a positive constant depending on $\alpha$, $d$ and $m$. On the other hand, by \eqref{2-4} and \eqref{2-10} we have 
$$G_1(\varepsilon,\Omega,0)=\Omega d\sin(x), \ \ \ \ \ \ \ G_2(\varepsilon,\Omega,0)=0$$
Hence it follows that
\begin{equation*}
	G_1(\varepsilon,\Omega,0)\neq 0,  \ \ \ \ \ \forall \, \varepsilon\in [-\varepsilon_0,\varepsilon_0]\setminus \{0\},
\end{equation*}
as long as $\varepsilon_0$ is chosen sufficiently small.

Let $\tilde f(x)=f(-x)$. Our last step is to show that if $(\varepsilon,f)$ is a solution to $\tilde G_1(\varepsilon,f)=0$, then $(-\varepsilon, \tilde f)$ is also a solution. By changing $y$ to $-y$ in the integral representation of $G^1(\varepsilon,\Omega,f)$ and using the fact $f$ is an even function, we obtain $\Omega_1(\varepsilon,f)=\Omega_1(-\varepsilon,\tilde f)$. Then we can insert it into $G^1(\varepsilon,\Omega,f)$ and derive $\tilde G_1(-\varepsilon, \tilde f)=0$ by a similar substitution of variables as before.

Hence the proof of Theorem \ref{thm1} is finished.

\subsection{Existence for the case $1<\alpha<2$}

To prove the regularity of $G^\alpha(\varepsilon, \Omega, f)$, we give an alternative characterization of the space $X^{k+\alpha-1}$.
\begin{lemma}\label{lem2-6}
	For $1<\alpha<2$ and $g(x)=\sum\limits_{j=2}^{\infty}a_j\cos(jx)$, $g\in X^{k+\alpha-1}$ if and only if
	\begin{equation*}
		g\in X^k, \ \ \ \left\|\int_0^{2\pi}\frac{\partial^kg(x-y)-\partial^kg(x)}{|\sin(\frac{y}{2})|^\alpha}dy\right\|_{L^2}<\infty.
	\end{equation*}
\end{lemma}
\begin{proof}
	By Lemma \ref{A-1} in Appendix and Plancherel's identity, it holds
	\begin{equation*}
		\left\|\int_0^{2\pi}\frac{\partial^kg(x-y)-\partial^kg(x)}{|\sin(\frac{y}{2})|^\alpha}dy\right\|_{L^2}^2=\sum\limits_{j=2}^{\infty}O(j^{\alpha-1})^2|\widehat{\partial^kg}_j|^2=\sum\limits_{j=2}^{\infty}O(j^{k+\alpha-1})^2|\hat{g}_j|^2,
	\end{equation*}
    from which we conclude the two characterizations of the space $X^{k+\alpha-1}$ are equivalent.
\end{proof}

Denote $V^r$ as the open neighborhood of zero in $X^{k+\alpha-1}$
\begin{equation*}
	V^r:=\left\{g\in X^{k+\alpha-1}: \ \|g\|_{X^{k+\alpha-1}}<r\right\}
\end{equation*}
with $0<r<1$ and $k\ge 3$. The following lemma is a version of Lemma \ref{lem2-2} for the case $1<\alpha<2$.
\begin{lemma}\label{lem2-7}
	$G^\alpha(\varepsilon, \Omega, f): \left(-\frac{1}{2}, \frac{1}{2}\right)\times \mathbb{R} \times V^r \rightarrow Y^{k-1}$ continuous.
\end{lemma}
\begin{proof}
	It is obvious that that $G_1: \left(-\frac{1}{2}, \frac{1}{2}\right)\times \mathbb{R} \times V^r \rightarrow Y^{k-1}$ is continuous, and can be rewritten as 
	\begin{equation}\label{2-20}
		G_1=\Omega \left(d\sin(x)+\varepsilon|\varepsilon|^\alpha\mathcal{R}_1(\varepsilon,f)\right),
	\end{equation}
	where $\mathcal{R}_1(\varepsilon,f): \left(-\frac{1}{2}, \frac{1}{2}\right)\times \mathbb{R} \times V^r \rightarrow Y^{k-1}$ is continuous. 
	
	Since $f\in V^r$, by changing $y$ to $-y$, we derive that $G_2(\varepsilon,f)$ is odd. Applying Taylor formula \eqref{2-5} on $G_{21}$, we can also show the possible singularity at $\varepsilon=0$ can not occur. From the proof of Lemma \ref{lem2-2} we know that the most singular term is $G_{22}$. For this reason, we can take $\partial^{k-1}$ derivatives of $G_{22}$
    \begin{equation*}
    	\begin{split}
    		\partial^{k-1}G_{22}&=C_\alpha\int\!\!\!\!\!\!\!\!\!\; {}-{}\frac{(\partial^kf(y)-\partial^kf(x))\cos(x-y)dy}{\left( |\varepsilon|^{2+2\alpha}\left(f(x)-f(y)\right)^2+4(1+\varepsilon|\varepsilon|^\alpha f(x))(1+\varepsilon|\varepsilon|^\alpha f(y))\sin^2\left(\frac{x-y}{2}\right)\right)^{\frac{\alpha}{2}}}\\	
    		& \ \ \ \ -\frac{\alpha C_\alpha}{2}\int\!\!\!\!\!\!\!\!\!\; {}-{}\frac{\cos(x-y)}{\left( |\varepsilon|^{2+2\alpha}\left(f(x)-f(y)\right)^2+4(1+\varepsilon|\varepsilon|^\alpha f(x))(1+\varepsilon|\varepsilon|^\alpha f(y))\sin^2\left(\frac{x-y}{2}\right)\right)^{1+\frac{\alpha}{2}}}\\
    		& \ \ \ \ \times\left(2(f(x)-f(y))(f'(x)-f'(y))+4(f(x)f'(y)+f'(x)f(y))\sin^2(\frac{x-y}{2})\right)\\
    		& \ \ \ \ \times(\partial^{k-1}f(y)-\partial^{k-1}f(x))dy +l.o.t,
    	\end{split}
    \end{equation*}
    where $l.o.t$ stands for the lower order derivative terms. Since $\|\partial^if\|_{L^\infty}\le C \|f\|_{X^{k+\alpha-1}}<\infty$ for $i=0,1,2$, by H\"older inequality and mean value theorem, we conclude that
    \begin{equation*}
    	\begin{split}
    		\left\|\partial^{k-1}G_{22}\right\|_{L^2}&\le C\left\|\int\!\!\!\!\!\!\!\!\!\; {}-{}\frac{\partial^kf(x)-\partial^kf(y)}{|\sin(\frac{x-y}{2})|^\alpha}dy\right\|_{L^2}+C\left\|\int\!\!\!\!\!\!\!\!\!\; {}-{}\frac{\partial^{k-1}f(x)-\partial^{k-1}f(y)}{|\sin(\frac{x-y}{2})|^\alpha}dy\right\|_{L^2}\\
    		&\le C \|f\|_{X^{k+\alpha-1}}+C\|f\|_{X^{k+\alpha-2}}<\infty,
    	\end{split}
    \end{equation*}
    where we use the alternative characterization of the space $X^{k+\alpha-1}$ in Lemma \ref{lem2-6}. So we have verified that the range of $G_2$ is in $Y^{k-1}$.
    
    By \eqref{2-7} and notations given in Lemma \ref{lem2-2}, we can prove the continuity for the most singular term $G_{22}$, namely, for $f_1,f_2\in V^r$, it holds $\|G_{22}(\varepsilon, f_1)-G_{22}(\varepsilon, f_2)\|_{Y^{k-1}}\le C\|f_1-f_2\|_{X^{k+\alpha-1}}$. Hence we claim that $G_2(\varepsilon, f): \left(-\frac{1}{2}, \frac{1}{2}\right)\times V^r \rightarrow Y^{k-1}$ is continuous. Using Taylor formula \eqref{2-5}, we collect the main terms of $G_2$ and obtain
    \begin{equation}\label{2-21}
    	G_2=C_\alpha(1-\frac{\alpha}{2})\int\!\!\!\!\!\!\!\!\!\; {}-{} \frac{f(x-y)\sin(y)dy}{\left(4\sin^2(\frac{y}{2})\right)^{\frac{\alpha}{2}}}-C_\alpha\int\!\!\!\!\!\!\!\!\!\; {}-{} \frac{(f'(x)-f'(x-y))\cos(y)dy}{\left(4\sin^2(\frac{y}{2})\right)^{\frac{\alpha}{2}}}+\varepsilon|\varepsilon|^\alpha \mathcal{R}_{2}(\varepsilon,f),
    \end{equation}
    where $\mathcal{R}_2(\varepsilon, f): \left(-\frac{1}{2}, \frac{1}{2}\right)\times V^r \rightarrow Y^{k-1}$ is continuous.
    
    At last, we have no difficulty showing that $G_3(\varepsilon, f): \left(-\frac{1}{2}, \frac{1}{2}\right)\times V^r \rightarrow Y^{k-1}$ is continuous. We can use Taylor formula \eqref{2-5} again to derive 
    \begin{equation}\label{2-22}
    	G_3=-\sum_{i=1}^{m-1} \frac{\alpha C_\alpha(-1+\cos(\frac{2\pi i}{m}))\sin (x)}{2\left((-1+\cos(\frac{2\pi i}{m}))^2 +\sin^2(\frac{2\pi i}{m})\right)^{1+\frac{\alpha}{2}}d^{1+\alpha}}+|\varepsilon|^\alpha\mathcal{R}_{3}(\varepsilon,f),
    \end{equation}
    where $\mathcal{R}_3(\varepsilon, f): \left(-\frac{1}{2}, \frac{1}{2}\right)\times V^r \rightarrow Y^{k-1}$ is continuous. So the proof is complete.    
\end{proof}

For $(\varepsilon,\Omega,f)\in \left(-\frac{1}{2}, \frac{1}{2}\right)\times \mathbb{R} \times V^r$ and $h\in X^{k+\alpha-1}$, we define
\begin{equation*}
	\partial_fG^\alpha(\varepsilon, \Omega, f)h:=\lim\limits_{t\to0}\frac{1}{t}\left(G^\alpha(\varepsilon, \Omega, f+th)-G^\alpha(\varepsilon, \Omega, f)\right)
\end{equation*}
to be the Gateaux derivative of $G^\alpha(\varepsilon, \Omega, f)$. The continuity of $\partial_fG^\alpha(\varepsilon, \Omega, f)h$ can be derived using a same method as $\alpha=1$. Thus we omit its proof, and just state it as follow. 
\begin{lemma}\label{lem2-8}
	For every $(\varepsilon,\Omega,f)\in \left(-\frac{1}{2}, \frac{1}{2}\right)\times \mathbb{R} \times V^r$, $\partial_fG^\alpha(\varepsilon, \Omega, f)h: X^{k+\alpha-1}\rightarrow Y^{k-1}$ is continuous.
\end{lemma}

Calculate as the case $\alpha=1$. One has
\begin{equation*}
	\partial_fG^\alpha(0,\Omega,0)h=C_\alpha(1-\frac{\alpha}{2})\int\!\!\!\!\!\!\!\!\!\; {}-{} \frac{h(x-y)\sin(y)dy}{\left(4\sin^2(\frac{y}{2})\right)^{\frac{\alpha}{2}}}-C_\alpha\int\!\!\!\!\!\!\!\!\!\; {}-{} \frac{(h'(x)-h'(x-y))\cos(y)dy}{\left(4\sin^2(\frac{y}{2})\right)^{\frac{\alpha}{2}}}.
\end{equation*}
We can obtain the corresponding result of Lemma \ref{lem2-4}:
\begin{lemma}\label{lem2-9}
	Let $h=\sum\limits_{j=2}^\infty a_j\cos(jx)$ be in $X^{k+\alpha-1}$. Then it holds
	\begin{equation*}
		\partial_fG^\alpha(0,\Omega,0)h=\sum\limits_{j=2}^\infty a_j \gamma_j j\sin(jx) 
	\end{equation*}
	with
	\begin{equation*}
		\gamma_j=2^{\alpha-1}\frac{ \Gamma(1-\alpha)}{(\Gamma(1-\frac{\alpha}{2}))^2}\left(\frac{\Gamma(1+\frac{\alpha}{2})}{\Gamma(2-\frac{\alpha}{2})}-\frac{\Gamma(j+\frac{\alpha}{2})}{\Gamma(1+j-\frac{\alpha}{2})}\right).
	\end{equation*}
	Moreover, for each $\Omega\in\mathbb R$, $\partial_fG^1(0,\Omega,0):X^{k+\alpha-1}\to Y_0^{k-1}$ is an isomorphism.
\end{lemma}
\begin{proof}
	The first half of this lemma is a direct result of \eqref{2-17} and \eqref{2-18}. 
	
	To prove $\partial_fG^1(0,\Omega,0):X^{k+\alpha-1}\to Y_0^{k-1}$ is an isomorphism, we notice that $\{\gamma_j\}$ is a monotone sequence with positive lower bound by Lemma \ref{A-1}. So the kernel of $\partial_fG^1(0,\Omega,0)$ is trival. It suffices to show that for each $p(x)\in Y_0^{k-1}$, there exists a $h(x) \in X^{k+\alpha-1}$ such that $\partial_fG^1(0,\Omega,0)h=p$. From the first half of this lemma, if $p$ is given by $p(x)=\sum\limits_{j=2}^\infty b_j\sin(jx)$, then $h$ satisfies
	\begin{equation*}
		h(x)=\sum\limits_{j=2}^\infty b_j\gamma_j^{-1}j^{-1}\cos(jx).
	\end{equation*}
	Since $\gamma_j=O(j^{\alpha-1})$ from Lemma \ref{A-1}, we deduce
	$$\|h\|_{X^{k+\alpha-1}}=\sum\limits_{j=2}^\infty b_j^2\gamma_j^{-2}j^{-2}j^{2k+2\alpha-2}\le C\sum\limits_{j=2}^\infty b_j^2j^{2k-2}\le C\|p\|_{Y_0^{k-1}},$$
	and the proof is complete.	
\end{proof}

According to \eqref{2-20} \eqref{2-21} \eqref{2-22}, if we denote
\begin{equation}\label{2-23}
	\Omega_\alpha^*:=\sum_{i=1}^{m-1} \frac{\alpha C_\alpha(-1+\cos(\frac{2\pi i}{m}))}{2\left((-1+\cos(\frac{2\pi i}{m}))^2 +\sin^2(\frac{2\pi i}{m})\right)^{1+\frac{\alpha}{2}}d^{2+\alpha}},
\end{equation}
then it holds $G^\alpha(0,\Omega_1^*,0)=0$. As the previous case $\alpha=1$, we are to restrict the range of $G^\alpha(\varepsilon,\Omega,f)$ in $Y_0^{k-1}$ by adjusting angular velocity $\Omega$.
\begin{lemma}\label{lem2-10}
	There exists
	\begin{equation*}
		\Omega_\alpha(\varepsilon, f):=\Omega_\alpha^*+\varepsilon^\alpha\mathcal{R}_\Omega(\varepsilon,f)
	\end{equation*}
	with $\Omega_\alpha^*$ given in \eqref{2-23} and  continuous function $\mathcal{R}_\Omega(\varepsilon,f):X^{k+\alpha-1}\to \mathbb{R}$, such that $\tilde G^\alpha(\varepsilon,f):(-\frac{1}{2},\frac{1}{2})\times V^r\to Y_0^{k-1}$ is given by
	\begin{equation*}
		\tilde G^\alpha(\varepsilon,f):=G^\alpha(\varepsilon,\Omega_1(\varepsilon,f), f).
	\end{equation*}
	Moreover, $\partial_f\mathcal{R}_\Omega(\varepsilon,f)h:X^{k+\alpha-1}\to \mathbb{R}$ is continuous.
\end{lemma}
\begin{proof}
	Similar to the proof of Lemma \ref{lem2-5}, by \eqref{2-20} \eqref{2-21} \eqref{2-22} we claim that $\Omega_\alpha(\varepsilon, f)$ must satisfy
	\begin{equation*}
		\Omega_\alpha\left(d+\varepsilon|\varepsilon|^\alpha\tilde{\mathcal{R}}_1\right)+\varepsilon|\varepsilon|^\alpha\tilde{\mathcal{R}}_2-\sum_{i=1}^{m-1} \frac{\alpha C_\alpha(-1+\cos(\frac{2\pi i}{m}))}{2\left((-1+\cos(\frac{2\pi i}{m}))^2 +\sin^2(\frac{2\pi i}{m})\right)^{1+\frac{\alpha}{2}}d^{1+\alpha}}+\varepsilon^\alpha\tilde{\mathcal{R}}_{3}=0
	\end{equation*}
	with $\tilde{\mathcal{R}}_i(\varepsilon,f)$ $(i=1,2,3)$ the contribution of $\mathcal{R}_i$ to the first Fourier coefficient.	Direct calculation yields
	\begin{equation*}
		\Omega_1(\varepsilon, f):=\Omega_1^*+\varepsilon\mathcal{R}_\Omega(\varepsilon,f)
	\end{equation*}
	with some continuous function $\mathcal{R}_\Omega(\varepsilon,f):X^{k+\alpha-1}\to \mathbb{R}$. Since $\partial_f{\mathcal{R}}_i(\varepsilon,f)$ $(i=1,2,3)$ are continuous, $\partial_f\mathcal{R}_\Omega(\varepsilon,f)h:X^{k+\alpha-1}\to \mathbb{R}$ is also continuous. Hence we have completed the proof.
\end{proof}

We are prepared to prove Theorem \ref{thm2}.

{\bf Proof of Theorem \ref{thm2}:}
Similar to the proof of Theorem \ref{thm1}, we can prove $\partial_f\tilde G^\alpha(0,0)h: X^{k+\alpha-1}\to Y^{k-1}_0$ is an isomorphism. Then by implicit function theorem, there exists $\varepsilon_0>0$ such that
\begin{equation*}
	\left\{(\varepsilon,f)\in [-\varepsilon_0,\varepsilon_0]\times V^r \ : \ \tilde G_1(\varepsilon,f)=0\right\}
\end{equation*}
is parametrized by one-dimensional curve $\varepsilon\in [-\varepsilon_0,\varepsilon_0]\to (\varepsilon, f_\varepsilon)$. Using the strategy for the case $\alpha=1$, we can prove that for each $\varepsilon\in[-\varepsilon_0,\varepsilon_0]\setminus \{0\}$, it always holds $f_\varepsilon\neq 0$. Moreover, by substitution of variables, it is easy to see that if $(\varepsilon,f)$ is a solution to $\tilde G_1(\varepsilon,f)=0$, then $(-\varepsilon, \tilde f)$ is also a solution with $\tilde f(x)=f(-x)$. So the proof is complete.\qed

\section{Existence of travelling global solutions for the gSQG equation}

As in the previous section, the point is to deal with $H^\alpha(\varepsilon, W, f)$ with the form
\begin{equation*}
	H^\alpha(\varepsilon, W, f)=H_1+H_2+H_3,
\end{equation*}
where
\begin{equation*}
	H_1=-W\left(\sin(x)-\frac{\varepsilon|\varepsilon|^\alpha f'(x)}{(1+\varepsilon|\varepsilon|^\alpha f(x))}\cos(x)\right),
\end{equation*}

\begin{equation*}
	\begin{split}
		H_2&=\frac{C_\alpha}{\varepsilon|\varepsilon|^{\alpha}}\int\!\!\!\!\!\!\!\!\!\; {}-{} \frac{(1+\varepsilon|\varepsilon|^\alpha f(y))\sin(x-y)dy}{\left( |\varepsilon|^{2+2\alpha}\left(f(x)-f(y)\right)^2+4(1+\varepsilon|\varepsilon|^\alpha f(x))(1+\varepsilon|\varepsilon|^\alpha f(y))\sin^2\left(\frac{x-y}{2}\right)\right)^{\frac{\alpha}{2}}}\\
		&+C_\alpha\int\!\!\!\!\!\!\!\!\!\; {}-{} \frac{(f'(y)-f'(x))\cos(x-y)dy}{\left( |\varepsilon|^{2+2\alpha}\left(f(x)-f(y)\right)^2+4(1+\varepsilon|\varepsilon|^\alpha f(x))(1+\varepsilon|\varepsilon|^\alpha f(y))\sin^2\left(\frac{x-y}{2}\right)\right)^{\frac{\alpha}{2}}}\\		
		&+\frac{C_\alpha f'(x)}{1+\varepsilon|\varepsilon|^\alpha f(x)}\int\!\!\!\!\!\!\!\!\!\; {}-{} \frac{(f(x)-f'(y))\cos(x-y)dy}{\left( |\varepsilon|^{2+2\alpha}\left(f(x)-f(y)\right)^2+4(1+\varepsilon|\varepsilon|^\alpha f(x))(1+\varepsilon|\varepsilon|^\alpha f(y))\sin^2\left(\frac{x-y}{2}\right)\right)^{\frac{\alpha}{2}}}\\	
		&+\frac{C_\alpha\varepsilon|\varepsilon|^\alpha}{1+\varepsilon|\varepsilon|^\alpha f(x)}\int\!\!\!\!\!\!\!\!\!\; {}-{} \frac{f'(x)f'(y)\sin(x-y)dy}{\left( |\varepsilon|^{2+2\alpha}\left(f(x)-f(y)\right)^2+4(1+\varepsilon|\varepsilon|^\alpha f(x))(1+\varepsilon|\varepsilon|^\alpha f(y))\sin^2\left(\frac{x-y}{2}\right)\right)^{\frac{\alpha}{2}}},\\	
	\end{split}
\end{equation*}
and
\begin{equation*}
	\begin{split}
		&H_3=\frac{C_\alpha}{\varepsilon}\int\!\!\!\!\!\!\!\!\!\; {}-{} \frac{(1+\varepsilon|\varepsilon|^\alpha f(y))\sin(x-y)dy}{\left|(R(x)\cos(x)+R(y)\cos(y)-2d)^2+(R(x)\sin(x)+R(y)\sin(y))^2\right|^{\frac{\alpha}{2}}}\\
		&+\frac{C_\alpha}{(1+\varepsilon|\varepsilon|^\alpha f(x))}\int\!\!\!\!\!\!\!\!\!\; {}-{} \frac{\varepsilon |\varepsilon|^{2\alpha}f'(x)f'(y)\sin(x-y)dy}{\left|(R(x)\cos(x)+R(y)\cos(y)-2d)^2+(R(x)\sin(x)+R(y)\sin(y))^2\right|^{\frac{\alpha}{2}}}\\
		&+\frac{C_\alpha}{(1+\varepsilon|\varepsilon|^\alpha f(x))}\int\!\!\!\!\!\!\!\!\!\; {}-{} \frac{(|\varepsilon|^\alpha f(x)f'(y)-|\varepsilon|^\alpha f(x)f'(y))\cos(x-y)dy}{\left|(R(x)\cos(x)+R(y)\cos(y)-2d)^2+(R(x)\sin(x)+R(y)\sin(y))^2\right|^{\frac{\alpha}{2}}}.\\
	\end{split}
\end{equation*}

Since the discussion of $H^\alpha(\varepsilon, W, f)$ is similar to $G^\alpha(\varepsilon, W, f)$ in Section 2, we only sketch the proof of Theorem \ref{thm2}: Notice that $H_2(\varepsilon,f)$ equals $G_2(\varepsilon,f)$ given in \eqref{2-2}, and $H_3(\varepsilon,f)$ equals $-G_3(\varepsilon,f)$ given in \eqref{2-3} when $m=2$. Thus we can use the same method as in Section 2 to prove the continuity of $H^\alpha(\varepsilon, W, f)$ and $\partial_fH^\alpha(\varepsilon, W, f)$. Moreover, one can verify that for each $W\in\mathbb{R}$, $\partial_fH^\alpha(0, W, 0)$ is an isomorphism from $X^{k}_{\log}$ $(\alpha=1)$, or $X^{k+\alpha-1}$ $(1<\alpha<2)$ to $Y^{k-1}_0$ by Lemma \ref{lem2-4} and Lemma \ref{lem2-9}.

Take $W_\alpha(\varepsilon,f)=W_\alpha^*+O(\varepsilon^\alpha)$ with 
$$W_\alpha^*= \frac{\alpha C_\alpha}{2(2d)^{1+\alpha}},$$
such that the range of $H^\alpha(\varepsilon, W_\alpha, f)$ is in $Y^{k-1}_0$. By applying implicit function theorem at $(0,W_\alpha,0)$, we can obtain the existence of travelling global solutions just as in Section 2.

\section{Regularity and convexity}
In this section, we show the regularity of vortices boundary and the convexity of $D_0^\varepsilon$. As mentioned in Section 3, the expression $H^\alpha(\varepsilon, W, f)$ shares the same regular properities  as $G^\alpha(\varepsilon, W, f)$. Therefore, we focus on proving the regularity of co-rotating solutions. The regularity of traveling solutions can be deduced in a same way, thus we omit the proof.
Recall that 
\begin{equation*}
	G^\alpha(\varepsilon, \Omega, f)=G_1+G_2+G_3,
\end{equation*}
where $G_1$, $G_2$ and $G_3$ is given by \eqref{2-1}, \eqref{2-2} and \eqref{2-3} respectively. We will take $k-1$ derivatives and divide the equation into the form 
\begin{equation}\label{4-1}
	L(\partial ^{k-1} f) +S(\partial^{k-1} f)=J(f),
\end{equation}
where $L$ is linear part of the most singular term and $L$ is invertible, $S$ has the same singularity as $L$ but $S$ is small and negligible, and $J$ is the remaining terms regular than $L$ and $S$. Then, we are able to improve the regularity of $f$ using bootstrap due to the difference of singularity on both sides of the equation \eqref{4-1} and the invertibility of $L+S$. For the case $\alpha\in(0,1)$, the proof of $C^\infty$ regularity has no difference with that in \cite{Cas1}. So we will begin with the case $\alpha=1$.

\subsection{The case $\alpha=1$}
This case is different from \cite{Cas1} mainly due to the existence of $\varepsilon$, which makes the linear part $L$ invertiable. We will use the following function space 
\begin{equation*}
	H^{k}_{\log}:=\left\{ g\in H^k,  \ \left\|\int_0^{2\pi}\frac{\partial^kg(x-y)-\partial^kg(x)}{|\sin(\frac{y}{2})|}dy\right\|_{L^2}<\infty \right\}
\end{equation*}
with the norm $||g||_{H^{k}_{\log}}=||g||_{H^k}+ \ \left\|\int_0^{2\pi}\frac{\partial^kg(x-y)-\partial^kg(x)}{|\sin(\frac{y}{2})|}dy\right\|_{L^2}$.

We choose $L$, $S$ and $J$ as follows.

\begin{equation}\label{4-2}
	L(\partial ^{k-1} f)=-C_\alpha\int\!\!\!\!\!\!\!\!\!\; {}-{} \frac{\cos(y)\left((\partial ^{k-1}f)'(x)-(\partial ^{k-1}f)'(x-y)\right)}{\left(4\sin^2\left(\frac{y}{2}\right)\right)^{\frac{1}{2}}}dy,
\end{equation}       
                                                                          
\begin{equation}\label{4-3}
	\begin{split}
	&\quad S(\partial ^{k-1} f)\\
	&=-\Omega\left(\varepsilon^{3}(\partial ^{k-1}f)'(x)-\frac{d\varepsilon^{2}(\partial ^{k-1}f)'(x)\cos(x)}{R(x)}\right)\\
	&-\frac{\varepsilon^{2}f'(x)}{R(x)}\int\!\!\!\!\!\!\!\!\!\; {}-{} \frac{\sin(y)(\partial ^{k-1}f)'(x-y)}{\left|\left(R(x)-R(x-y)\right)^2+4R(x)R(x-y)\sin^2\left(\frac{y}{2}\right)\right|^{\frac{1}{2}}}dy\\
	&-\frac{\varepsilon^{2}(\partial ^{k-1}f)'(x)}{ R(x)}\int\!\!\!\!\!\!\!\!\!\; {}-{} \frac{\sin(y)f'(x-y)+\cos(y)\left(f(x)-f(x-y)\right)}{\left|\left(R(x)-R(x-y)\right)^2+4R(x)R(x-y)\sin^2\left(\frac{y}{2}\right)\right|^{\frac{1}{2}}}dy\\
	&-\sum_{i=1}^{m-1} \frac{\varepsilon^{3}f'(x)}{R(x)} \int\!\!\!\!\!\!\!\!\!\; {}-{} \frac{(\partial ^{k-1}f)'(y)\sin(x-y-\frac{2\pi i}{m})dy}{\left| \left(\boldsymbol{z}(x)-(d,0)\right)-Q_{\frac{2\pi i}{m}}\left(\boldsymbol{z}(y)-(d,0)\right)\right|}\\
	&-\sum_{i=1}^{m-1} \frac{\varepsilon^{3}(\partial ^{k-1}f)'(x)}{ R(x)} \int\!\!\!\!\!\!\!\!\!\; {}-{}\frac{f'(y)\sin(x-y-\frac{2\pi i}{m})+\cos(x-y-\frac{2\pi i}{m})\left(f(x)-f(y)\right)dy}{\left| \left(\boldsymbol{z}(x)-(d,0)\right)-Q_{\frac{2\pi i}{m}}\left(\boldsymbol{z}(y)-(d,0)\right)\right|}\\
	&-\sum_{i=1}^{m-1} \varepsilon\int\!\!\!\!\!\!\!\!\!\; {}-{} \frac{ \cos(x-y-\frac{2\pi i}{m})\left((\partial ^{k-1}f)'(x)-(\partial ^{k-1}f)'(y)\right)dy}{\left| \left(\boldsymbol{z}(x)-(d,0)\right)-Q_{\frac{2\pi i}{m}}\left(\boldsymbol{z}(y)-(d,0)\right)\right|}\\
	&-\int\!\!\!\!\!\!\!\!\!\; {}-{} \frac{\cos(y)\left((\partial ^{k-1}f)'(x)-(\partial ^{k-1}f)'(x-y)\right)}{\left(4\sin^2\left(\frac{y}{2}\right)\right)^{\frac{1}{2}}}\times \left[\frac{1}{\left| \left(\frac{R(x)-R(x-y)}{2\sin\left(\frac{y}{2}\right)}\right)^2+R(x)R(x-y)\right|^{\frac{1}{2}}}-1\right]dy\\
	&=S_1+S_2+S_3+S_4+S_5+S_6+S_7.
	\end{split}
\end{equation}

\begin{equation}\label{4-4}
	\begin{split}
	&\quad J(f)\\
	&=-\Omega\partial ^{k-1}\left(\frac{dR'(x)\cos(x)}{R(x)}+d\sin(x)\right)+\Omega\left(\frac{d\varepsilon^{2}(\partial ^{k-1}f)'(x)\cos(x)}{R(x)}\right)\\
	&+\partial ^{k-1}\left(\frac{f'(x)}{R(x)}\int\!\!\!\!\!\!\!\!\!\; {}-{} \frac{R'(x-y)\sin(y)+[R(x)-R(x-y)]\cos(y)}{\left| \left(R(x)-R(x-y)\right)^2+4R(x)R(x-y)\sin^2\left(\frac{y}{2}\right)\right|^{\frac{1}{2}}}dy\right)\\
	&\quad -\frac{\varepsilon^{2}f'(x)}{R(x)}\int\!\!\!\!\!\!\!\!\!\; {}-{} \frac{\sin(y)(\partial ^{k-1}f)'(x-y)}{\left|\left(R(x)-R(x-y)\right)^2+4R(x)R(x-y)\sin^2\left(\frac{y}{2}\right)\right|^{\frac{1}{2}}}dy\\
	&\quad -\frac{\varepsilon^{2}(\partial ^{k-1}f)'(x)}{R(x)}\int\!\!\!\!\!\!\!\!\!\; {}-{} \frac{\sin(y)f'(x-y)+\cos(y)\left(f(x)-f(x-y)\right)}{\left|\left(R(x)-R(x-y)\right)^2+4R(x)R(x-y)\sin^2\left(\frac{y}{2}\right)\right|^{\frac{1}{2}}}dy\\
	&+\partial ^{k-1}\left(\frac{1}{R(x)\varepsilon^{2}}\int\!\!\!\!\!\!\!\!\!\; {}-{} \frac{R(x-y)\sin(y)}{\left| \left(R(x)-R(x-y)\right)^2+4R(x)R(x-y)\sin^2\left(\frac{y}{2}\right)\right|^{\frac{1}{2}}}dy\right)\\
	&+\partial ^{k-1} \left( \int\!\!\!\!\!\!\!\!\!\; {}-{} \frac{[f'(x-y)-f'(x)]\cos(y)}{\left| \left(R(x)-R(x-y)\right)^2+4R(x)R(x-y)\sin^2\left(\frac{y}{2}\right)\right|^{\frac{1}{2}}} dy \right)\\
	&\quad- \int\!\!\!\!\!\!\!\!\!\; {}-{} \frac{[\partial ^{k}f(x-y)- \partial ^{k}f(x)]\cos(y)}{\left| \left(R(x)-R(x-y)\right)^2+4R(x)R(x-y)\sin^2\left(\frac{y}{2}\right)\right|^{\frac{1}{2}}} dy\\
	&+\partial ^{k-1}G_3\\
	&\quad -\sum_{i=1}^{m-1} \frac{\varepsilon^{3}f'(x)}{R(x)} \int\!\!\!\!\!\!\!\!\!\; {}-{} \frac{(\partial ^{k-1}f)'(y)\sin(x-y-\frac{2\pi i}{m})dy}{\left| \left(\boldsymbol{z}(x)-(d,0)\right)-Q_{\frac{2\pi i}{m}}\left(\boldsymbol{z}(y)-(d,0)\right)\right|}\\
	&\quad -\sum_{i=1}^{m-1} \varepsilon\int\!\!\!\!\!\!\!\!\!\; {}-{} \frac{ \cos(x-y-\frac{2\pi i}{m})\left((\partial ^{k-1}f)'(x)-(\partial ^{k-1}f)'(y)\right)dy}{\left| \left(\boldsymbol{z}(x)-(d,0)\right)-Q_{\frac{2\pi i}{m}}\left(\boldsymbol{z}(y)-(d,0)\right)\right|}\\
	&\quad -\sum_{i=1}^{m-1} \frac{\varepsilon^{3}(\partial ^{k-1}f)'(x)}{R(x)} \int\!\!\!\!\!\!\!\!\!\; {}-{} \frac{f'(y)\sin(x-y-\frac{2\pi i}{m})+\cos(x-y-\frac{2\pi i}{m})\left(f(x)-f(y)\right)dy}{\left| \left(\boldsymbol{z}(x)-(d,0)\right)-Q_{\frac{2\pi i}{m}}\left(\boldsymbol{z}(y)-(d,0)\right)\right|}\\
	&=J_1(f)+J_2(f)+J_3(f)+J_4(f)+J_5(f).
	\end{split}
\end{equation}
\begin{lemma}\label{lem4-1}
	$L$ is linear and invertible and maps $H^{2}_{\log}$ to $H^{1}$.
\end{lemma}
\begin{proof}
	It is obvious that $L$ is linear and maps $H^{2}_{\log}$ to $H^{1}$. We only need to prove $L$ is invertible. For any $p\in H^1$, without loss of generality, we assume
	\begin{equation*}
		p(x)=\sum_{j=1}^\infty \left(a_j\sin(jx)+b_j\cos(jx)\right).
	\end{equation*}
	We define $q$ to be 
	\begin{equation*}
		q(x)=\sum_{j=1}^\infty \left(\bar{a}_j\sin(jx)+\bar{b}_j\cos(jx)\right),
	\end{equation*}
	where $\bar{a}_j$ and $\bar{b}_j$ is given by
	\begin{align*}
		\bar{a}_j=\frac{b_j}{2j\left(2\sum_{i=1}^j\frac{1}{2i-1}-(2+\frac{1}{2j-1}+\frac{1}{2j+1})\right)},\\
		\bar{b}_j=\frac{b_j}{2j\left(2\sum_{i=1}^j\frac{1}{2i-1}-(2+\frac{1}{2j-1}-\frac{1}{2j+1})\right)}.
	\end{align*}
	It is easy to check that $0<|\bar{a}_j|, |\bar{b}_j|<\infty$ and hence $\bar{a}_j, \bar{b}_j$ are well-defined. Note that $|\bar{a}_j|\sim\frac{|a_j|}{j\log j}$ and $|\bar{b}_j|\sim\frac{|b_j|}{j\log j}$ for $j$ large, thus $q\in H^{2}_{\log}$. Using Lemma 2.6 \cite{Cas1}, we easily deduce $L(q)=p$. The uniqueness of $q$ follows easily.
\end{proof}
\begin{lemma}\label{lem4-2}
	$||S(\partial ^{k-1} f)||_{H^1}\leq C(\varepsilon) ||\partial ^{k-1} f||_{H^{2}_{\log}}$, where $C(\varepsilon) \to0$ as $\varepsilon\to 0$.
\end{lemma}
\begin{proof}
	It is obvious $||S_i||_{H^1}\leq C\varepsilon ||f||_{H^{2}_{\log}}$ for $i=1,2,3,4,5,6$.  Note that the remaining term $S_7$ is almost the same as the term $S_2$ in Lemma 4.8 \cite{Cas1}, and hence by very similar argument as in P.978 of \cite{Cas1}, one has $||S_7||_{H^1}\leq C(\varepsilon) ||f||_{H^{2}_{\log}}$ with $C(\varepsilon) \to 0$ as $\varepsilon\to 0$.
\end{proof}

\begin{lemma}\label{lem4-3}
	If $f\in X^{k}_{\log}$, then $J(f)\in H^{1}$.
\end{lemma}
\begin{proof}
	Taking one more derivative to $J(f)$, we compute the most singular terms for each $g_i$ with $i=1,2,3,4,5$.\\
	1. $$\partial J_1(f)=\Omega\varepsilon^2 \frac{f'(x)\partial^{k}f(x)}{R^2(x)}+l.o.t.$$
	Obviously, $ \partial J_1(f)\in L^2$ and hence $g_1\in H^1$.\\
	2. The most singular terms for $\partial J_2$ are
	\begin{equation*}
		J_{21}=\frac{\varepsilon f'(x)}{ R(x)}\int\!\!\!\!\!\!\!\!\!\; {}-{} \frac{(\partial ^{k}f(x)-\partial ^{k}f(x-y))\cos(y)}{\left| \left(R(x)-R(x-y)\right)^2+4R(x)R(x-y)\sin^2\left(\frac{y}{2}\right)\right|^{\frac{1}{2}}}dy,
	\end{equation*}
	\begin{align*}
		J_{22}&=-\frac{\varepsilon f'(x)}{2 R(x)}\int\!\!\!\!\!\!\!\!\!\; {}-{}  \frac{R'(x-y)\sin(y)+[R(x)-R(x-y)]\cos(y)}{\left| \left(R(x)-R(x-y)\right)^2+4R(x)R(x-y)\sin^2\left(\frac{y}{2}\right)\right|^{\frac{3}{2}}}\\
		&\quad\times\Big( 2(R(x)-R(x-y))(\partial ^{k}f(x)-\partial ^{k}f(x-y)\Big.\\
		&\qquad\left.+4\left(\partial ^{k}f(x)R(x-y)+R(x)\partial ^{k}f(x-y)\right)\sin^2\left(\frac{y}{2}\right)\right)dy,
	\end{align*}
	The kernal in $J_{21}$ can be split into  $$\frac{\cos(y)}{|\left(R(x)-R(x-y)\right)^2+4R(x)R(x-y)\sin^2\left(\frac{y}{2}\right)|^{\frac{1}{2}}}=\frac{1}{(R'(x)^2+R(x)^2)^{\frac{1}{2}}}\frac{1}{\left(4\sin^2\left(\frac{y}{2}\right)\right)^{\frac{1}{2}}}+\mathcal{L}_{21}(x,y),$$
	where $\mathcal{L}(x,y)$ is regular and belongs to $L^2$. Now $J_{21}$ becomes
	\begin{align*}
		J_{21}&=\frac{\varepsilon f'(x)}{ R(x)(R'(x)^2+R(x)^2)^{\frac{1}{2}}}\int\!\!\!\!\!\!\!\!\!\; {}-{}  \frac{(\partial ^{k}f(x)-\partial ^{k}f(x-y))}{\left(4\sin^2\left(\frac{y}{2}\right)\right)^{\frac{1}{2}}}dy\\
		&+\frac{\varepsilon f'(x)}{2\pi R(x)}\int_0^{2\pi} (\partial ^{k}f(x)-\partial ^{k}f(x-y))\mathcal{L}_{21}(x,y) dy.	
	\end{align*}
	Thus, by the defination of $H^{k}_{\log}$, we have $||J_{21}||_{L^2}\leq C||f||_{H^{k}_{\log}}$.
	
	Similarly, by Taylor expansion, we can divide the kernal in the first part of $J_{22}$ into
	\begin{align*}
		&\quad \frac{(R'(x-y)\sin(y)+[R(x)-R(x-y)]\cos(y))(R(x)-R(x-y))}{\left| \left(R(x)-R(x-y)\right)^2+4R(x)R(x-y)\sin^2\left(\frac{y}{2}\right)\right|^{\frac{3}{2}}}\\
		&=\frac{\left(2R'(x)+R'(x)\right)R'(x)}{(R'(x)^2+R(x)^2)^{\frac{3}{2}}}\frac{1}{\left(4\sin^2\left(\frac{y}{2}\right)\right)^{\frac{1}{2}}}+{\mathcal{L}}_{22}(x,y),
	\end{align*}
	where ${\mathcal{L}}(x,y)$ is regular and belongs to $L^2$.  Thus, the first part of $J_{22}$ belongs to $L^2$.
	As for the remaining part of $J_{22}$, it is enough to notice that 
	$$\left|\left|\frac{\left(R'(x-y)\sin(y)+[R(x)-R(x-y)]\cos(y)\right)\sin^2\left(\frac{y}{2}\right)}{\left| \left(R(x)-R(x-y)\right)^2+4R(x)R(x-y)\sin^2\left(\frac{y}{2}\right)\right|^{\frac{3}{2}}}\right|\right|_{L^\infty}\leq C. $$
	Thus, we conclude that $J_{22}\in L^2$, which implies $J_2\in H^1$.\\
	3. The most singular terms in $\partial J_3$ are 
	$$J_{31}=-\frac{\partial^kf(x)}{ R(x)^2}\int\!\!\!\!\!\!\!\!\!\; {}-{}  \frac{R(x-y)\sin(y)}{\left| \left(R(x)-R(x-y)\right)^2+4R(x)R(x-y)\sin^2\left(\frac{y}{2}\right)\right|^{\frac{1}{2}}}dy,$$
	$$J_{32}=\frac{1}{ R(x)}\int\!\!\!\!\!\!\!\!\!\; {}-{}  \frac{\partial^kf(x-y)\sin(y)}{\left| \left(R(x)-R(x-y)\right)^2+4R(x)R(x-y)\sin^2\left(\frac{y}{2}\right)\right|^{\frac{1}{2}}}dy$$
	\begin{align*}
		J_{33}&=-\frac{1}{2 R(x)}\int\!\!\!\!\!\!\!\!\!\; {}-{}  \frac{R(x-y)\sin(y)}{\left| \left(R(x)-R(x-y)\right)^2+4R(x)R(x-y)\sin^2\left(\frac{y}{2}\right)\right|^{\frac{3}{2}}}\\&\quad\times\Big( 2(R(x)-R(x-y))(\partial ^{k}f(x)-\partial ^{k}f(x-y)\Big.\\
		&\qquad\left.+4\left(\partial ^{k}f(x)R(x-y)+R(x)\partial ^{k}f(x-y)\right)\sin^2\left(\frac{y}{2}\right)\right)dy.
	\end{align*}
	The kernal $\frac{\sin(y)}{\left| \left(R(x)-R(x-y)\right)^2+4R(x)R(x-y)\sin^2\left(\frac{y}{2}\right)\right|^{\frac{1}{2}}}$ in $J_{31}$ and $J_{32}$ is regular, thus $J_{31}, J_{32}\in L^2$. $J_{33}$ is more regular than $J_{22}$ and hence belongs to $L^2$.\\
	4. The most singular term for $\partial J_4$ is 
	\begin{align*}
		\partial J_{4}&=-\frac{C_\alpha}{4\pi}\int_0^{2\pi} \frac{(f'(x-y)-f'(x))\cos(y)}{\left| \left(R(x)-R(x-y)\right)^2+4R(x)R(x-y)\sin^2\left(\frac{y}{2}\right)\right|^{\frac{3}{2}}}\\&\quad\times\Big( 2(R(x)-R(x-y))(\partial ^{k}f(x)-\partial ^{k}f(x-y)\Big.\\
		&\qquad\left.+4\left(\partial ^{k}f(x)R(x-y)+R(x)\partial ^{k}f(x-y)\right)\sin^2\left(\frac{y}{2}\right)\right)dy+l.o.t,
	\end{align*}
	which can also be considered in the similar spirit of $J_{22}$.\\
	5. The highest derivative of $f$ which appears in $\partial J_5$ is $\partial^k f$, and for $i\ge1$, the denominator $\left| \left(\boldsymbol{z}(x)-(d,0)\right)-Q_{\frac{2\pi i}{m}}\left(\boldsymbol{z}(y)-(d,0)\right)\right|^\alpha$ has a positive lower bound. Then $\partial J_5\in L^2$ follows easily.
\end{proof}

If we combine the above lemmas, we obtain the following regularity theorem.
\begin{corollary}\label{coro4-4}
	If $f\in H^{k}_{\log}$ solves $G^\alpha(\varepsilon, \Omega, f)=0$ for some $\varepsilon$ small and $\Omega$, then $f\in H^{k+1}_{\log}$ for any $k\geq 3$.
\end{corollary}
In view of the above theorem, by bootstrap argument, we eventually obtain the regularity of solution constracted in Section 2.1.

\begin{theorem}\label{thm4-5}
   When $\alpha=1$, the solution $f$ to $G^\alpha(\varepsilon, \Omega, f)=0$ lies in $H^k$ for any $k\geq 3$ and hence $f\in C^\infty$.
\end{theorem}

\subsection{The case $1<\alpha<2$}
We choose $L$, $S$ and $J$ as follows.

\begin{equation}\label{4-5}
	L(\partial ^{k-1} f)=-C_\alpha\int\!\!\!\!\!\!\!\!\!\; {}-{} \frac{(\partial ^{k-1}f)'(x)-(\partial ^{k-1}f)'(x-y)}{\left(4\sin^2\left(\frac{y}{2}\right)\right)^{\frac{\alpha}{2}}}dy,
\end{equation}           
                                                                      
\begin{equation}\label{4-6}
	S(\partial ^{k-1} f)=-C_\alpha\int\!\!\!\!\!\!\!\!\!\; {}-{} \frac{(\partial ^{k-1}f)'(x)-(\partial ^{k-1}f)'(x-y)}{\left(4\sin^2\left(\frac{y}{2}\right)\right)^{\frac{\alpha}{2}}}\times \left[\frac{1}{\left( R'(x)^2+R(x)^2\right)^{\frac{\alpha}{2}}}-1\right]dy,
\end{equation}

\begin{equation}\label{4-7}
	\begin{split}
	J(f)&=\Omega\left(\varepsilon \partial ^{k}R(x)-\partial ^{k-1}\left(\frac{dR'(x)\cos(x)}{R(x)}+d\sin(x)\right)\right)\\
	&+\partial ^{k-1}\left(\frac{C_\alpha f'(x)}{R(x)}\int\!\!\!\!\!\!\!\!\!\; {}-{} \frac{R'(x-y)\sin(y)+[R(x)-R(x-y)]\cos(y)}{\left| \left(R(x)-R(x-y)\right)^2+4R(x)R(x-y)\sin^2\left(\frac{y}{2}\right)\right|^{\frac{\alpha}{2}}}dy\right)\\
	&+\partial ^{k-1}\left(\frac{C_\alpha}{ R(x)\varepsilon|\varepsilon|^{\alpha}}\int\!\!\!\!\!\!\!\!\!\; {}-{} \frac{R(x-y)\sin(y)}{\left| \left(R(x)-R(x-y)\right)^2+4R(x)R(x-y)\sin^2\left(\frac{y}{2}\right)\right|^{\frac{\alpha}{2}}}dy\right)\\
	&+\partial ^{k-1} \left( C_\alpha\int\!\!\!\!\!\!\!\!\!\; {}-{} \frac{[f'(x-y)-f'(x)]\cos(y)}{\left| \left(R(x)-R(x-y)\right)^2+4R(x)R(x-y)\sin^2\left(\frac{y}{2}\right)\right|^{\frac{\alpha}{2}}} dy \right)\\
	&\quad- C_\alpha\int\!\!\!\!\!\!\!\!\!\; {}-{} \frac{[\partial ^{k}f(x-y)- \partial ^{k}f(x)]\cos(y)}{\left| \left(R(x)-R(x-y)\right)^2+4R(x)R(x-y)\sin^2\left(\frac{y}{2}\right)\right|^{\frac{\alpha}{2}}} dy\\
	&+C_\alpha\int\!\!\!\!\!\!\!\!\!\; {}-{} \frac{[\partial ^{k}f(x-y)- \partial ^{k}f(x)]}{\left(4\sin^2\left(\frac{y}{2}\right)\right)^{\frac{\alpha}{2}}} \times\left[ \frac{\cos(y)}{\left| \left(\frac{R(x)-R(x-y)}{2\sin\left(\frac{y}{2}\right)}\right)^2+4R(x)R(x-y)\right|^{\frac{\alpha}{2}}} - \frac{1}{\left| 4(R'(x)^2+R(x)^2)\right|^{\frac{\alpha}{2}}}  \right]dy\\
	&+\partial ^{k-1}G_3\\
	&=J_1(f)+J_2(f)+J_3(f)+J_4(f)+J_5(f)+J_6(f).
	\end{split}
\end{equation}

\begin{lemma}[Lemma 4.5, \cite{Cas1}]\label{lem4-6}
	$L$ is linear and invertible and maps $H^{2\alpha-1}$ to $H^{\alpha-1}$ .
\end{lemma}

\begin{lemma}\label{lem4-7}
	$||S(\partial ^{k-1} f)||_{H^{\alpha-1}}\leq C_\varepsilon ||\partial ^{k-1} f||_{H^{2\alpha-1}}$, where $C(\varepsilon) \to 0$ as $\varepsilon\rightarrow 0$.
\end{lemma}
\begin{proof}
	Denote $$s_1(x,y)= C_\alpha \frac{(\partial ^{k-1}f)'(x)-(\partial ^{k-1}f)'(x-y)}{\left(4\sin^2\left(\frac{x-y}{2}\right)\right)^{\frac{\alpha}{2}}},$$ and $$s_2(x)=\frac{1}{\left( R'(x)^2+R(x)^2\right)^{\frac{\alpha}{2}}}-1.$$
	By the mean value theorem, we have 
	\begin{align*}
		s_2(x)&=-\frac{\alpha}{2}\delta^{-1-\frac{\alpha}{2}}\left(R'(x)^2+R(x)^2-1\right)\\
		&=-\frac{\alpha}{2}\delta^{-1-\frac{\alpha}{2}}\varepsilon|\varepsilon|^{\alpha}\left(2f(x)+\varepsilon|\varepsilon|^{\alpha}\left(f'(x)^2+f(x)^2\right)\right),
	\end{align*}
	where $\delta$ is between 1 and $R'(x)^2+R(x)^2$.
	Thus, we deduce
	\begin{equation*}
		||s_2(x)||_{L^\infty}\leq C|\varepsilon|^{1+\alpha},
	\end{equation*}
	where $C$ is a constant depending on $\alpha$ and $||f||_{C^1}$.  Note that by imbedding theorem, it holds $||f||_{C^1}\leq C||f||_{H^2}< \infty.$
	Direct calculation gives
	\begin{equation*}
		||s_2(x)||_{H^{\frac{3}{2}}}\leq C|\varepsilon|^{1+\alpha},
	\end{equation*}
	where $C$ is a constant depending on $\alpha$ and $||f||_{H^{\frac{5}{2}}}\leq ||f||_{H^3}\leq 1$.
	Using Lemma 4.2 \cite{Cas1} with $\sigma=\frac{1}{2}+2-\alpha$, we obtain
	$||S(\partial ^{k-1} f)||_{H^{\alpha-1}}\leq C(\varepsilon) ||\partial ^{k-1} f||_{H^{2\alpha-1}}$ for some $C(\varepsilon)$ satisfying $C(\varepsilon) \to 0$ as $\varepsilon\to 0$.	
\end{proof}

\begin{lemma}\label{lem4-8}
	If $f\in H^{k+\alpha-1}$, then $J(f)\in H^{\alpha-1}$.
\end{lemma}
\begin{proof}
	
	Obviously, $||J_1(f)||_{H^{\alpha-1}}\leq C(||f||_{k+\alpha-1}+1)$. The proof of boundedness of terms $J_2$, $J_4$ and $J_5$ in $H^{\alpha-1}$ are exactly the same as  \cite{Cas1}, so we omit the details. Compared with $G_3$ in Section 4.2 \cite{Cas1}, there is an $\varepsilon|\varepsilon|^{\alpha}$ in the denominator of our $J_3$. However, notice that $R=1+\varepsilon|\varepsilon|^{\alpha}f$ implies $\partial^j R =\varepsilon|\varepsilon|^{\alpha}\partial^j f$ for all $j\geq1$. Thus, the $\varepsilon|\varepsilon|^{\alpha}$ in the denominator of $J_3$ can be eliminated and hence it will not cause any singularity. Therefore, the proof in \cite{Cas1} still work here. 
	
	For the remaining term $g_6$, arguing as above, we find that the $\varepsilon$ in the denominator of $g_6$ can be eliminated and hence it will not cause singularity. Moreover, for $\varepsilon$ small and $i\ge1$, $\left| \left(\boldsymbol{z}(x)-(d,0)\right)-Q_{\frac{2\pi i}{m}}\left(\boldsymbol{z}(y)-(d,0)\right)\right|$ has a positive lower bound, which means the integral is regular, and hence one can easily verify the  boundedness of  $g_6$ in $H^{\alpha-1}$.
\end{proof}

As a conclusion of above lemmas, we obtain the following regularity theorem.
\begin{corollary}\label{coro4-9}
	If $f\in H^{k+\alpha-1}$ solves $G^\alpha(\varepsilon, \Omega, f)=0$ for some $\varepsilon$ small and $\Omega$, then $f\in H^{k+2\alpha-1}$ for any $k\geq 3$.
\end{corollary}
In view of the above theorem, by bootstrap argument, we eventually obtain the regularity of solution for the case $1<\alpha<2$.

\begin{theorem}\label{thm4-10}
	When $1<\alpha<2$, the solution $f$ to $G^\alpha(\varepsilon, \Omega, f)=0$ lies in $H^k$ for any $k\geq 3$ and hence $f\in C^\infty$.
\end{theorem}

At the end of this section, we show the convexity of $D_0^\varepsilon$ by calculating the curvature.
\begin{theorem}\label{thm4-11}
	For $\varepsilon$ sufficiently small, $R(x)=1+\varepsilon|\varepsilon|^\alpha f(x)$ parameterizes convex patches.
\end{theorem}
\begin{proof}
	We show this by computing the signed curvature. Given $x\in[0,2\pi)$, the signed curvature at $x$ is
	\begin{align*}
		\varepsilon \kappa(x)=\frac{R(x)^2+2R'(x)^2-R(x)R''(x)}{\left(R(x)^2+R'(x)^2\right)^{\frac{3}{2}}}=\frac{1+O(\varepsilon)}{1+O(\varepsilon)}>0,
	\end{align*}
	for $\varepsilon$ small, which implies the convexity.
\end{proof}

\appendix

\section{Auxiliary results}
We list some auxiliary results for the usage in the preceding sections. For $0<\alpha<2$ and $j\in \mathbb{N}^+$, let
\begin{equation*}
	\mathcal{S}_j(x)=\int_0^{2\pi}\frac{\sin(jx)-\sin(jx-jy)}{\left(\sin(\frac{y}{2})\right)^\alpha}dy,
\end{equation*} 
\begin{equation*}
	\mathcal{C}_j(x)=\int_0^{2\pi}\frac{\cos(jx)-\cos(jx-jy)}{\left(\sin(\frac{y}{2})\right)^\alpha}dy.
\end{equation*} 
Using the fundamental properties of Euler gamma function, we can calculate both $\mathcal{S}_j(x)$ and $\mathcal{C}_j(x)$ to be  trigonometric polynomials, and investigate the asymptotic behavior for their coefficients; see Lemma 2.6 in \cite{Cas1} and Formula 6.1.46 in \cite{Ab}. We conclude these results in the following lemma. 
\begin{lemma}\label{A-1}
    For $0<\alpha<2$ and $j\in \mathbb{N}^+$, $\mathcal{S}_j(x)$ and $\mathcal{C}_j(x)$ satisfy
    \begin{equation*}
		\mathcal{S}_j(x)=\beta_j\sin(jx), \ \ \ \ \  \mathcal{C}_j(x)=\beta_j\cos(jx),
    \end{equation*}
    where if $\alpha\neq 1$,
    \begin{equation*}
		\beta_j=2^\alpha\frac{2\pi\Gamma(1-\alpha)}{\Gamma(\frac{\alpha}{2})\Gamma(1-\frac{\alpha}{2})}\left(\frac{\Gamma(\frac{\alpha}{2})}{\Gamma(1-\frac{\alpha}{2})}-\frac{\Gamma(j+\frac{\alpha}{2})}{\Gamma(j+1-\frac{\alpha}{2})}\right),
    \end{equation*}
    and if $\alpha=1$,
    \begin{equation*}
    	\beta_j=\sum\limits^j_{i=1}\frac{8}{2i-1}.
    \end{equation*}
    Moreover, $\{\beta_j\}$ is increasing with respect to $j$, and has following asymptotic behavior when $j$ is large: if $\alpha=1$, $\beta_j=O(\ln j)$; if $\alpha>1$, $\beta_j=O(j^{\alpha-1})$.
\end{lemma}

\phantom{s}
\thispagestyle{empty}

\end{document}